\documentclass[reqno]{amsart}
\usepackage{amsmath,amsthm,amsfonts,amssymb,amscd}
\usepackage{enumerate}
\usepackage[vcentermath,enableskew]{youngtab}
\usepackage{microtype}
\usepackage{tikz}
\usepackage{mathrsfs}
\usepackage[ocgcolorlinks, linkcolor=red!80!black, citecolor=blue!80!black]{hyperref}

\newtheorem{thm}{Theorem}[section]
\newtheorem*{thm*}{Theorem}
\newtheorem{lem}[thm]{Lemma}
\theoremstyle{definition} \newtheorem{defn}[thm]{Definition}
\newtheorem*{defn*}{Definition}
\theoremstyle{definition} \newtheorem{ex}[thm]{Example}
\newtheorem*{lem*}{Lemma}
\newtheorem{cor}[thm]{Corollary}
\newtheorem*{cor*}{Corollary}
\theoremstyle{definition} \newtheorem{rem}[thm]{Remark}

\newtheorem*{conj*}{Conjecture}

\newtheorem{prop}[thm]{Proposition}

\newcommand{\CC}{\mathbb{C}}
\newcommand{\Cyl}{\mathcal{C}}

\newcommand{\NN}{\mathbb{N}}
\newcommand{\ZZ}{\mathbb{Z}}
\newcommand{\RR}{\mathbb{R}}

\newcommand{\Cor}{\mathcal{K}}
\newcommand{\Sym}{\operatorname{Sym}}

\newcommand{\sq}{{\scriptstyle \square}}

\DeclareMathOperator{\Des}{Des}

\DeclareMathOperator{\sgn}{sgn}
\DeclareMathOperator{\Bound}{Bound}
\DeclareMathOperator{\trunc}{trunc}
\DeclareMathOperator{\ch}{ch}
\DeclareMathOperator{\fch}{fch}

\DeclareMathOperator{\Res}{Res}

\DeclareMathOperator{\BCov}{BCov}

\DeclareMathOperator{\Gr}{Gr}
\DeclareMathOperator{\av}{av}

\DeclareMathOperator{\lex}{lex}
\DeclareMathOperator{\stb}{stb}
\DeclareMathOperator{\GL}{GL}
\DeclareMathOperator{\diag}{diag}
\DeclareMathOperator{\tr}{tr}
\DeclareMathOperator{\vspan}{span}
\DeclareMathOperator{\Hom}{Hom}
\DeclareMathOperator{\Fl}{Fl}
\DeclareMathOperator{\id}{id}

\renewcommand{\det}{\operatorname{det}}

\newcommand{\trncotimes}{\bar\otimes}
\DeclareMathOperator{\trch}{\overline{ch}}

\newcommand{\affS}{\widetilde{S}}
\newcommand{\affF}{\widetilde{F}}
\newcommand{\fS}{\mathfrak{S}}
\renewcommand{\SS}{\mathbb{S}}

\begin{document}

\author{Brendan Pawlowski}
\title{A representation-theoretic interpretation of positroid classes}

\begin{abstract}
A positroid is the matroid of a real matrix with nonnegative maximal minors, a positroid variety is the closure of the locus of points in a complex Grassmannian whose matroid is a fixed positroid, and a positroid class is the cohomology class Poincar\'e dual to a positroid variety. We define a family of representations of general linear groups whose characters are symmetric polynomials representing positroid classes. These representations are certain diagram Schur modules in the sense of James and Peel. This gives a new algebraic interpretation of the Schubert structure constants for the product of a Schubert polynomial and Schur polynomial, and of the 3-point Gromov-Witten invariants for Grassmannians, proving a conjecture of Postnikov. As a byproduct, we obtain an effective algorithm for decomposing positroid classes into Schubert classes.
\end{abstract}

\maketitle

\section{Introduction}

Suppose $A$ is a $k \times n$ matrix over a field $K$, and $I$ is a $k$-subset of $[n] := \{1, 2, \ldots, n\}$. The \emph{$I$th Pl\"ucker coordinate of $A$}, for which we write $p_I(A)$, is the $k \times k$-minor of $A$ in columns $I$. The \emph{matroid} of $A$ is then $\{I \in {[n] \choose k} : p_I(A) \neq 0\}$. The matroid of $A$ depends only on the row span of $A$, so in fact a $k$-dimensional subspace of $K^n$ has a well-defined matroid. Let $\Gr_K(k,n)$ denote the Grassmannian of $k$-planes in $K^n$, or simply $\Gr(k,n)$ in the case $K = \CC$.

\begin{defn}
The \emph{totally nonnegative Grassmannian} $\Gr_{\RR}(k,n)^+$ is the set of $k$-planes $V = \operatorname{rowspan}(A)$ where all the maximal minors of $A$ are nonnegative. A \emph{positroid} is the matroid of a member of $\Gr_{\RR}(k,n)^+$.
\end{defn}

Postnikov \cite{postnikov-positroids} gave several combinatorial objects which are in bijection with positroids, and used them to describe the locus of points in $\Gr_{\RR}(k,n)^+$ whose matroid is a fixed positroid. Knutson, Lam, and Speyer \cite{positroidjuggling} studied the following complex analogue of Postnikov's positroid cells.

\begin{defn}
The \emph{positroid variety} $\Pi_M \subseteq \Gr(k,n)$ of a positroid $M \subseteq {[n] \choose k}$ is the Zariski closure of the set of $k$-planes with matroid $M$.
\end{defn}

Let $\Lambda(k)$ denote the ring of symmetric polynomials $\ZZ[x_1, \ldots, x_k]^{S_k}$, and $\Lambda^{n-k}(k)$ the quotient by the ideal generated by the homogeneous symmetric polynomials
\begin{equation*}
h_d(X_k) = \sum_{1 \leq i_1 \leq \cdots \leq i_d \leq k} x_{i_1} \cdots x_{i_d}
\end{equation*}
for $d > n-k$. The ring $\Lambda^{n-k}(k)$ is isomorphic to the integral cohomology ring of the Grassmannian $\Gr(k, n)$ of $k$-planes in $\CC^n$, and under this isomorphism the cone of cohomology classes Poincar\'e dual to subvarieties of $\Gr(k, n)$ corresponds to the cone of Schur-positive elements in $\Lambda^{n-k}(k)$. 

\begin{defn}
For a positroid $M \subseteq {[n] \choose k}$, let $G_M \in \Lambda^{n-k}(k)$ represent the cohomology class of the positroid variety $\Pi_M$.
\end{defn}
The Schur-positive elements $G_M \in \Lambda^{n-k}(k)$ are the central objects of this paper. Via taking characters, the Grothendieck ring of finite-dimensional complex polynomial representations of $\GL(\CC^k)$ is isomorphic to  $\Lambda(k)$, and our main result writes $G_M$ as the character of a certain representation of $\GL(\CC^k)$ (rather, its image in $\Lambda^{n-k}(k)$).

\begin{defn}
A \emph{diagram} is a finite subset of $\ZZ^2$.
\end{defn}

Let $S_d$ be the symmetric group on $d$ letters. To a diagram $D$ of size $d$ one associates an element $y_D$ of the group algebra $\CC[S_d]$, the \emph{Young symmetrizer} of $c_D$. If $V$ is a complex vector space, $S_d$ acts on $V^{\otimes d}$ on the right by permuting tensor factors, while $\GL(V)$ acts diagonally on the left. The \emph{generalized Schur module} $V[D]$ associated to $D$ is then the left $\GL(V)$-module $V^{\otimes d} y_D$. In cases which are understood, such as \cite{magyarborelweil, percentavoiding, liuforests, liubranching}, the algebraic properties of $V[D]$ tend to reflect interesting combinatorics relating to the diagram $D$, but there is no general combinatorial description of the irreducible decomposition of $V[D]$ or its character. In Section~\ref{sec:applications}, we observe that our results give, as a byproduct, a combinatorial algorithm for decomposing $V[D]$ into irreducibles when $D$ has at most $3$ rows.

\begin{ex} \label{ex:schur-examples} \hfill
\begin{enumerate}[(a)]
\item If $D$ is the Young diagram of a partition $\lambda$, then $V[D]$ is irreducible with character the Schur polynomial $s_{\lambda}(x_1, \ldots, x_k)$, where $k = \dim V$.

\item More generally, if $D$ is a skew Young diagram $\lambda \setminus \mu$, the character of $V[D]$ is the skew Schur polynomial $s_{\lambda \setminus \mu}(x_1, \ldots, x_k)$, and $V[\lambda \setminus \mu] \simeq \bigoplus_{\nu} c_{\lambda \mu}^{\nu} V[\nu]$, where $c_{\lambda \mu}^{\nu}$ is a Littlewood-Richardson coefficient. The dimension of $V[\lambda \setminus \mu]$ is the number of semistandard tableaux of shape $\lambda \setminus \mu$ on $[k]$, while the dimension of the $1^d$-weight space is the number of standard tableaux (if $k \geq d = |\lambda\setminus \mu|$).

\item The \emph{Rothe diagram} of a permutation $w \in S_n$ is
\begin{equation*}
D(w) = \{(i, w(j)) \in [n] \times [n] : i < j, w(i) > w(j)\}.
\end{equation*}
It follows from \cite{kraskiewicz, plactification} that the character of $V[D(w)]$ is the \emph{Stanley symmetric function} $F_w(x_1, \ldots, x_k)$ (see Definition~\ref{defn:affine-stanley} below). The dimension of $V[D(w)]$ is then the number of column-strict balanced labellings of $D(w)$ on $[k]$ in the sense of \cite{balancedlabellings}. The dimension of its $1^{\ell(w)}$-weight space is the number of \emph{reduced words} for $w$: minimal sequences $i_1, \ldots, i_\ell$ such that $w = s_{i_1} \cdots s_{i_{\ell}}$ where $s_i = (i\,\,i+1) \in S_n$ (if $k \geq \ell(w)$). When $\ell(\lambda) < k$, the multiplicity of $V[\lambda]$ in $V[D(w)]$ is the number of semistandard tableaux of shape $\lambda$ whose column word is a reduced word for $w$, by \cite{edelman-greene}.
\end{enumerate}
\end{ex}

Our main result can be viewed as a generalization of Example~\ref{ex:schur-examples}(c) as follows. To each positroid $M \subseteq {[n] \choose k}$, Knutson-Lam-Speyer \cite{positroidjuggling} associate a certain \emph{affine permutation} $f_M$, i.e. a bijection $\ZZ \to \ZZ$ satisfying the quasi-periodicity property that $f_M(i+n) = f_M(i)+n$ for all $i$. They also prove that $G_M$ is the image in $\Lambda^{n-k}(k)$ of the \emph{affine Stanley symmetric function} $\affF_{f_M}$. Now define the Rothe diagram of $f_M$ as in Example~\ref{ex:schur-examples}(c), but viewed as a finite subset of the cylinder $\ZZ^2 / \ZZ(n,n)$ (this will cause no difficulties in defining the Schur module $V[D(f_M)]$).

\begin{thm} \label{thm:main}
For any positroid $M \subseteq {[n] \choose k}$, the image of the character of $V[D(f_M)]$ in $\Lambda^{n-k}(k)$ is $G_M$.
\end{thm}
In Section~\ref{sec:truncated-schur}, we discuss a slightly modified notion of Schur module whose characters can naturally be thought of as members of $\Lambda^{n-k}(k)$.

Besides being certain intersection numbers for $\Pi_M$, the Schur coefficients of $G_M$ have many other interpretations, and Theorem~\ref{thm:main} provides an algebraic proof of the nonnegativity of these integers. All of the following can be described as certain Schur coefficients of $G_M$:
\begin{enumerate}[(a)]
\item The 3-point Gromov-Witten invariants for $\Gr(k,n)$ \cite{postnikov-affine-quantum-schubert, lam-affine-stanley, positroidjuggling};
\item The Schubert coefficients in the product of a Schubert polynomial and a Schur polynomial (Proposition~\ref{prop:schubert-coefficients});
\item The Schur coefficients of the image of an affine Stanley symmetric function in $\Lambda^{n-k}(k)$ \cite{positroidjuggling};
\item The Schur coefficients of the symmetric function $\sum_c Q_{\Des(c)}$, where $Q_D$ is Gessel's fundamental quasisymmetric function and $c$ runs over maximal chains in an interval in Bergeron and Sottile's $k$-Bruhat order \cite{bergeron-sottile-skew-schubert, assaf-bergeron-sottile}.
\end{enumerate}
To elaborate on (a), Postnikov defined certain finite subsets $D$ of the cylinder $\ZZ^2/\ZZ(k,n-k)$ called \emph{toric skew shapes}, and associated to $D$ the \emph{toric Schur polynomial} $s_D(x_1, \ldots, x_k)$, the weight-generating function for semistandard fillings of $D$. He showed that the Schur coefficients of toric Schur polynomials are the 3-point Gromov-Witten invariants for $\Gr(k,n)$, and gave a conjectural representation-theoretic interpretation. 
\begin{conj*}[\cite{postnikov-affine-quantum-schubert}, Conjecture 10.1] For a toric skew shape $D \subseteq \ZZ^2/\ZZ(k,n-k)$, the toric Schur polynomial of $D$ is the character of $V[D]$, where $\dim V = k$. \end{conj*}
We will see (Theorem~\ref{thm:toric-schur-module}) that Postnikov's conjecture follows as a special case of Theorem~\ref{thm:main}.

There is an effective recursion for computing the Schur expansion of a Stanley symmetric function $F_w$ based on the so-called \emph{transition formula} of Lascoux and Sch\"utzenberger \cite{lascouxschutzenbergertree}. They construct a tree of permutations with root $w$ such that $F_v = \sum_{v'} F_{v'}$ for any node $v$ where $v'$ runs over children of $v$, and show that any sufficiently long path from the root leads to a node $v$ such that $F_v$ is a single, easily-described Schur function.

Lam and Shimozono \cite{affine-little-bump} prove analogous formulas for affine Stanley symmetric functions, but in the affine case it is unclear how to arrange these formulas into a recursion terminating in simple base cases. We show (Theorem~\ref{thm:LS-formula}) that these difficulties disappear upon passing to the quotient $\Lambda^{n-k}(k)$. That is, there is an effective (but no longer positive!) recursion for computing the Schur expansion of $G_M$ in the style of Lascoux and Sch\"utzenberger. The proof of Theorem~\ref{thm:main} proceeds by constructing a filtration of $V[D(f_M)]$ which, at the level of characters, matches this recursion for $G_M$ (similar arguments appear in \cite{billey-pawlowski-2013} and \cite{kraskiewicz-pragacz}).

In Section~\ref{sec:background}, we recall some background on affine permutations, symmetric functions, and the representation theory of $\GL(V)$, and prove some preparatory lemmas. In Section~\ref{sec:recurrences}, we prove an analogue of Lascoux-Sch\"utzenberger's transition formula for the symmetric functions $G_M$. Sections~\ref{sec:schur} and \ref{sec:truncated-schur} are devoted to Schur modules, and contain the main technical tools we use to relate the combinatorics of a diagram $D$ to $V[D]$. We then apply these tools to Schur modules of affine Rothe diagrams in Section~\ref{sec:main} and prove Theorem~\ref{thm:main}. Finally, Section~\ref{sec:applications} describes some applications of our results, including a proof of Postnikov's conjecture on toric Schur modules.

\subsection*{Acknowledgements}

I'm grateful to Thomas Lam and Vic Reiner for many helpful discussions, and for introducing me to several of the problems considered here. Ricky Liu and Alex Postnikov also deserve credit for useful conversations.

\section{Background} \label{sec:background}

\subsection{Affine permutations}

Let $n$ be a positive integer. An \emph{affine permutation of quasi-period $n$} is a bijection $f : \ZZ \to \ZZ$ satisfying $f(i+n) = f(i)+n$ for all $i$. We write $\affS_n$ for the group of all affine permutations. We will specify an affine permutation $f$ by the word $f(1), \ldots, f(n)$, since this uniquely determines $f$, writing $\overline{x}$ for $-x$. For instance, $f = 645\overline{1}$ is the affine permutation with $f(1) = 6$, $f(4) = -1$, $f(8) = 3$, and so on.

\begin{defn}
Given $k, m \in \ZZ$, let $\Cyl_{k,m}$ denote the cylinder $\ZZ^2 / \ZZ(k,m)$.
\end{defn}
It is natural to view the graph of $f \in \affS_n$ as a subset of $\Cyl_{n,n}$, namely $\{(i,f(i)) \in \Cyl_{n,n} : i \in \ZZ\}$. An \emph{inversion} of $f \in \affS_n$ is a point $(i,j) \in \Cyl_{n,n}$ such that $i < j$ and $f(i) > f(j)$. Let $\ell(f)$ denote the number of inversions of $f$. It is not hard to see that $\ell(f)$ is finite, for instance by verifying the specific formula
\begin{equation*}
\ell(f) = \sum_{\substack{1 \leq i < j \leq n \\ f(i) > f(j)}} \left \lceil \frac{f(j)-f(i)}{n} \right\rceil.
\end{equation*}

For $i \in \ZZ$, let $s_i \in \affS_n$ be the transposition interchanging $i+pn$ and $i+1+pn$ for all $p \in \ZZ$ and fixing all other integers. Let $\tau \in \affS_n$ be the shift map $\tau : i \mapsto i+1$. Starting from any $f \in \affS_n$, one can repeatedly multiply by adjacent transpositions $s_i$ and eventually obtain an affine permutation with no inversions, which is necessarily a power of $\tau$. Defining $\affS^0_n := \langle s_0, s_1, \ldots, s_{n-1} \rangle$, we see that $f$ may be written uniquely in the form $\tau^j g$ with $j \in \ZZ$ and $g \in \affS^0_n$. Let $\av(f) = j$. We state some basic properties of the map $\av$ without proof.

\begin{prop} \label{prop:av-properties} For $f \in \affS_n$,
\begin{enumerate}[(a)]
\item $\av(f) = \frac{1}{n}\sum_{i=1}^n (f(i)-i)$;
\item There is a unique expression $f = w + n\lambda$ where $w \in \langle s_1, \ldots, s_n \rangle \simeq S_n$ and $\lambda \in \ZZ^\ZZ$, and $\av(f) = \sum_{i=1}^n \lambda_i$;
\item $\av : \affS_n \to \ZZ$ is a group homomorphism, and $\affS_n$ is the semidirect product $\affS_n^0 \rtimes \langle \tau \rangle$.
\end{enumerate}
\end{prop}

The group $\affS^0_n$ is a Coxeter group (the affine Weyl group of type $\widetilde{A}_n$), and we will implicitly extend many notions of Coxeter theory from $g \in \affS^0_n$ to $\tau^j g$ for any $j \in \ZZ$: Coxeter length, descents, reduced words, Bruhat order, and so on. Note that $\ell(f)$ as defined above is in fact the Coxeter length of $f$. The subgroup $\langle s_1, \ldots, s_{n-1} \rangle$ of $\affS_n$ is isomorphic to the ordinary permutation group on $n$ letters $S_n$, and we will frequently identify the two.

\begin{defn} \label{defn:bounded}
An affine permutation $f \in \affS_n$ is \emph{bounded} if $i \leq f(i) \leq i+n$ for all $i \in \ZZ$. Let $\Bound(n) \subseteq \affS_n$ denote the subset of bounded permutations, and $\Bound(k,n)$ the set of $f \in \Bound(n)$ such that exactly $k$ of $f(1), \ldots, f(n)$ exceed $n$.
\end{defn}

\begin{lem} \label{lem:shifted-bound-k-n} $\Bound(k,n) = \Bound(n) \cap \tau^k \affS_n^0$. \end{lem}
\begin{proof}
Writing $f = w + n\lambda$ with $w \in S_n$, if $f$ is in $\Bound(n)$ then $\{\lambda_1, \ldots, \lambda_n\} \subseteq \{0,1\}$, and then $f \in \Bound(k,n)$ where $k = \sum_{i=1}^n \lambda_i$. Proposition~\ref{prop:av-properties}(b) now says $\av(f) = k$.
\end{proof}

By \cite[Theorem 3.1]{positroidjuggling}, $\Bound(k,n)$ is in bijection with the rank $k$ positroids on $[n]$. Another collection of objects in bijection with positroids which will be useful for us relies on the notion of $k$-Bruhat order. Given integers $i < j$ for which $i \not\equiv j \pmod{n}$, let $t_{ij} \in \affS^0_n$ be the transposition interchanging $i+pn$ and $j+pn$ for all $p \in \ZZ$. We use $<$ to denote the (strong) Bruhat order on $\affS_n$, i.e. the partial order with covering relations $f \lessdot ft_{ij}$ whenever $\ell(ft_{ij}) = \ell(f)+1$.
\begin{defn}[\cite{k-bruhat-order}]
For an integer $k$, the \emph{$k$-Bruhat order on $S_n$} is the partial order $\leq_k$ with covering relations $w \lessdot_k wt_{ij}$ if $i \leq k < j$ and $\ell(wt_{ij}) = \ell(w)+1$.
\end{defn}
For instance, $24315 \lessdot_3 24513$ but $24315 \not\hspace{-3pt}\lessdot_3 42315$. Let $[u,v]_k$ denote the $k$-Bruhat interval $\{u' \in S_n : u \leq_k u' \leq_k v\}$.
\begin{defn}
Define an equivalence relation on the set of $k$-Bruhat intervals in $S_n$ by declaring $[u,v]_k$ and $[ux, vx]_k$ to be equivalent if $\ell(ux)-\ell(u) = \ell(vx)-\ell(v) = \ell(x)$ where $x \in S_n$ stabilizes $[k]$. Let $Q(k,n)$ denote the set of equivalence classes of $k$-Bruhat intervals in $S_n$ under this equivalence relation, with $\langle u, v \rangle_k$ denoting the class of $[u,v]_k$. Partially order $Q(k,n)$ so that $I \leq I'$ if $I$ and $I'$ have representatives $[u, v]_k$ and $[u', v']_k$ such that $[u',v']_k \subseteq [u,v]_k$.
\end{defn}
For instance, $[24315, 24513]_3$ is equivalent to $[42315, 42513]_3$ via $x = s_1$, but not to $[24135, 24153]_3$ since $x = s_3$ does not stabilize $\{1,2,3\}$. Let $g_{k,n} \in \Bound(k,n)$ be the affine permutation $(n+1)\cdots (n+k)(k+1)\cdots n$. In the next theorem, we view $\affS_n$ and the subset $\Bound(k,n)$ as posets ordered by strong Bruhat order.

\begin{thm}[\cite{positroidjuggling}, Section 3] \label{thm:k-bruhat-bijection} The map $Q(k,n) \to \affS_n$ defined by $\langle u, v \rangle_k \mapsto f_{u,v} := ug_{k,n}v^{-1}$ is a well-defined injection of posets with image $\Bound(k,n)$, and $\ell(f_{u,v}) = k(n-k) - \ell(v) + \ell(u)$.
\end{thm}

For any integer $r$, we can consider $r$-Bruhat order on $\affS_n$ or $\Bound(k,n)$ with the same definition as for $S_n$. In the affine case, these partial orders are all isomorphic, because $\tau t_{ij} \tau^{-1} = t_{i+1,j+1}$, and so $f \leq_r g$ if and only if $\tau f \tau^{-1} \leq_{r+1} \tau g \tau^{-1}$. For this reason we focus on the case $r = 0$.
\begin{lem} \label{lem:0-bruhat-bijection} Under the bijection $\langle u, v \rangle_k \mapsto f_{u,v}$ of Theorem~\ref{thm:k-bruhat-bijection}, the $0$-Bruhat order on $\Bound(k,n)$ corresponds to the order on $Q(k,n)$ where $I \lessdot I'$ if $I$ and $I'$ have representatives $[u,v]_k$ and $[u',v]_k$ with $u \lessdot_k u'$. \end{lem}
\begin{proof}
Suppose $u \lessdot_k u' = ut_{ij} \leq_k v$, so $1 \leq i \leq k < j$. Then $g_{k,n}^{-1}(i) = i-n$ and $g_{k,n}^{-1}(j) = j$, so
\begin{equation} \label{eq:k-bruhat-cover}
f_{u',v} = u t_{ij} g_{k,n} v^{-1} = u g_{k,n} t_{i-n,j} v^{-1} = f_{u,v} t_{v(i)-n, v(j)}.
\end{equation}
Theorem~\ref{thm:k-bruhat-bijection} implies $f_{u,v} \lessdot f_{u',v}$, and since $v(i)-n \leq 0 < v(j)$ we see $f_{u,v} \lessdot_0 f_{u',v}$.

Conversely, suppose $f_{u,v} \lessdot f_{u,v}t_{i-n,j}$ where $i, j \in [n]$ and $f_{u,v}t_{i-n,j} \in \Bound(k,n)$. We can assume that $f_{u,v}t_{i-n,j} = f_{u',v'}$ where $[u',v']_k \subseteq [u,v]_k$, so either $u' = u$ or $v' = v$. If $u' = u$, then $f_{u,v}t_{i-n,j} = f_{u',v'}$ implies $v' = t_{i-n,j}v$, but then $v' \notin S_n$. Hence $u \lessdot_k u' \leq_k v$.
\end{proof}

\subsection{Symmetric functions and polynomials} \label{subsec:symm}
We use the following conventions for partitions. A partition $\lambda$ is a weakly decreasing sequence $\lambda = (\lambda_1 \geq \lambda_2 \geq \cdots \geq \lambda_{\ell(\lambda)} > 0)$ with $\lambda_i = 0$ for $i > \ell(\lambda)$. Alternatively, $(a_1^{k_1}, \ldots, a_m^{k_m})$ denotes the partition in which each part $a_i$ appears with multiplicity $k_i$. We draw Young diagrams in the French style, so $\lambda$ has Young diagram
\begin{equation*}
\{(-i,j) : \text{$1 \leq i \leq \ell(\lambda)$ and $1 \leq j \leq \lambda_{i}$}\},
\end{equation*}
which we view as a subset of $\ZZ^2$ or some cylinder $\Cyl_{k,m}$ depending on context. Given partitions $\lambda$ and $\mu$, say $\lambda \subseteq \mu$ if the Young diagram of $\mu$ contains the Young diagram of $\lambda$; equivalently, if $\lambda_i \leq \mu_i$ for all $i$.

Let $\Lambda$ be the ring of symmetric functions over $\ZZ$, and $\Lambda(k)$ the ring of symmetric polynomials $\ZZ[x_1, \ldots, x_k]^{S_k}$. Given $F \in \Lambda$ we write $F(X_k)$ for the polynomial $F(x_1, \ldots, x_k) \in \Lambda(k)$. For $m \in \ZZ$, let $J(m) \subseteq \Lambda(k)$ be the ideal $\langle h_d(X_k) : d > m \rangle$. The Jacobi-Trudi formula shows that $J(m) = \operatorname{span} \{s_{\lambda}(X_k) : \lambda \not\subseteq (m^k)\}$, where $s_{\lambda}$ is a Schur function.
\begin{defn}
Let $\Lambda^m(k) := \Lambda(k)/J(m)$, with $\trunc_{k,m}$ denoting the quotient map. We will also write $\trunc_{k,m}$ for the ring map $\Lambda \to \Lambda^m(k)$ sending $F$ to $\trunc_{k,m} F(X_k)$, and sometimes we will write simply $\bar F$ for $\trunc_{k,m} F$ when $k$ and $m$ are clear from context.
\end{defn}

\subsection{Polynomial representations of $\GL(V)$} \label{subsec:rep}
Fix a finite-dimensional complex vector space $V$, and let $T(V) = \bigoplus_{d=0}^{\infty} V^{\otimes d}$ denote the tensor algebra on $V$. Let $R(V)$ be the Grothendieck group of $\GL(V)$-submodules of $T(V)$: the free abelian group on isomorphism classes $[U]$ of submodules\footnote{For submodules $U, U' \subseteq T(V)$ it need not be the case that $U \oplus U' \hookrightarrow T(V)$, but we will still write $[U \oplus U']$ for $[U] + [U']$.} $U \subseteq T(V)$ modulo the relations $[U \oplus U'] = [U] + [U']$. The tensor product makes $R(V)$ into a ring.

Say $\dim V = k$. Recall that the \emph{character} of a complex representation $\rho : \GL(V) \to \GL(U)$ of $\GL(V)$ is the function $\ch(U) : (x_1, \ldots, x_k) \mapsto \tr \rho(\diag(x_1, \ldots, x_k))$, where $\diag(x_1, \ldots, x_k) \in \GL(V)$ is the diagonal matrix with diagonal entries $x_1, \ldots, x_k$, having chosen a basis for $V$. Suppose $U$ is a \emph{polynomial representation}, meaning that upon choosing bases, the entries of the matrices $\rho(g)$ for $g \in \GL(V)$ are polynomials in the entries of $g$. Then $\ch(U)$ is a polynomial, and in fact $\ch(U) \in \Lambda(k)$. The next theorem summarizes some basic facts about complex representations of $\GL(V)$ and their characters.
\begin{thm}[\cite{youngtableaux}] \label{thm:rep-theory-facts} \hfill
\begin{enumerate}[(a)]
\item The character map $\ch : R(V) \to  \Lambda(k)$ is a ring isomorphism.
\item The irreducible polynomial representations of $\GL(V)$ are $\ch^{-1} s_{\lambda}(X_k)$ for $\lambda$ such that $\ell(\lambda) \leq k$. In particular, they all occur as submodules of $T(V)$.
\end{enumerate}
\end{thm}
Theorem~\ref{thm:rep-theory-facts}(a) implies that $R(V)$ modulo the ideal spanned by all $[\ch^{-1} s_{\lambda}(X_k)]$ for $\lambda \not\subseteq (m^k)$ is isomorphic to the ring $\Lambda^{m}(k)$. In Section~\ref{sec:truncated-schur}, we construct this quotient of $R(V)$ in a more natural way and use it to define an appropriate notion of character which naturally maps into $\Lambda^m(k)$.

\subsection{Affine Stanley symmetric functions}

\begin{defn} \hfill
\begin{enumerate}[(a)]
\item A sequence $i_1, \ldots, i_k$ in $\ZZ/n\ZZ$ is \emph{cyclically decreasing} if
    \begin{enumerate}[(1)]
    \item all its entries are distinct, and;
    \item whenever $j$ and $j+1$ both appear in the sequence, $j+1$ appears before $j$.
    \end{enumerate}

\item An affine permutation $f$ is cyclically decreasing if $f = s_{i_1} \cdots s_{i_k}$ for some cyclically decreasing sequence $i_1, \ldots, i_k$.
\end{enumerate}
\end{defn}
\begin{ex}
$s_1 s_0 s_3 \in \tilde{S}_4$ is cyclically decreasing, but $s_1 s_3 s_0$ and $s_3 s_0 s_3$ are not.
\end{ex}

A factorization $f = f_1 \cdots f_p$ where $f, f_1, \ldots, f_p \in \affS_n^0$ is \emph{length-additive} if $\ell(f) = \sum_{i=1}^p \ell(f_i)$. 
\begin{defn}[\cite{lam-affine-stanley}] \label{defn:affine-stanley}
The \emph{affine Stanley symmetric function} of $f \in \affS_n$ is the power series
\begin{equation*}
\affF_f = \sum_{(f_1, \ldots, f_p)} x_1^{\ell(f_1)} \cdots x_p^{\ell(f_p)}
\end{equation*}
running over all length-additive factorizations $\tau^{-\av(f)} f = f_1 \cdots f_p$ such that each $f_i$ is cyclically decreasing.
\end{defn}

Definition~\ref{defn:affine-stanley} is due to Lam \cite{lam-affine-stanley}, who proved many basic properties about $\affF_f$ including the non-obvious fact that $\affF_f \in \Lambda$. When $f \in \langle s_1, \ldots, s_{n-1} \rangle \simeq S_n$, affine Stanley symmetric functions agree with the symmetric functions introduced by Stanley in \cite{stanleysymm} (except that Stanley's $G_w$ is our $\affF_{w^{-1}}$). Observe that the coefficient of a squarefree monomial in $\affF_f$ is the number of reduced words of $f$.

For $f \in S_n$, the results of \cite{edelman-greene} imply that $\affF_f$ is Schur-positive, but this need not hold for general affine $f$; for instance, $\affF_{5274} = s_{22} + s_{211} - s_{1111}$. However, it turns out that a predictable subset of the Schur coefficients of $\affF_f$ are nonnegative.
\begin{defn}
Given $f \in \affS_n$ and $0 \leq k \leq n$, define $G_{f,k} = \trunc_{k,n-k} \affF_f \in \Lambda^{n-k}(k)$. We usually suppress the dependence on $k$ and simply write $G_f$.
\end{defn}

\begin{thm}[\cite{positroidjuggling}, Theorem 7.1] \label{thm:schur-positive-truncation} If $f \in \Bound(k,n)$, then $G_{f,k} \in \Lambda^{n-k}(k)$ represents the cohomology class of a positroid variety. In particular, it is Schur-positive and nonzero. \end{thm}
Let $T = \langle \tau \rangle$. Since $\affF_f = \affF_{\tau f}$, it holds more generally that $G_{f,k}$ is Schur-positive and nonzero whenever $f \in T \Bound(k,n)$. Our next goal is to show that this is a necessary condition: that if $f \notin T \Bound(k,n)$, then $G_{f,k} = 0$.

\begin{defn} \label{defn:rothe} The \emph{Rothe diagram} of $f \in \affS_n$ is the set
\begin{equation*}
D(f) = \{(i,f(j)) \in \Cyl_{n,n} : i < j, f(i) > f(j)\}.
\end{equation*}
\end{defn}

\begin{defn}
The \emph{code} of $f \in \affS_n$ is the sequence $c(f) = (\ldots, c_1(f), \ldots, c_n(f), \ldots) \in \ZZ^\ZZ$ where $c_i(f) = \#\{i < i' : f(i) > f(i')\}$. Note that $c(f)$ is periodic with period $n$.
\end{defn}

\begin{ex} \label{ex:rothe}
We will draw diagrams in matrix coordinates, using $\square$ for points in $D(f)$ and $\cdot$ for points not in $D(f)$. As a visual aid we also draw the graph of $f$, using $\times$ for its members: the points in $D(f)$ are then exactly those which are strictly left of and above an $\times$. With these conventions,
\begin{equation*}
D(5274) = \begin{array}{ccccccccccccc}
\ddots  &       &       &       &       &         &       &       &        &       &       &       &\\
& \cdot & \square & \cdot & \square & \times   & \cdot & \cdot & \cdot & \cdot & \cdot & \cdot &\\
& \cdot & \times & \cdot & \cdot & \cdot   & \cdot & \cdot & \cdot & \cdot & \cdot & \cdot &\\
& \cdot & \cdot & \cdot & \square & \cdot   & \square & \times & \cdot & \cdot & \cdot & \cdot &\\
& \cdot & \cdot & \cdot & \times & \cdot   & \cdot & \cdot & \cdot & \cdot & \cdot & \cdot\\
& \cdot & \cdot & \cdot & \cdot & \cdot   & \square & \cdot & \square & \times & \cdot & \cdot &\\
& \cdot & \cdot & \cdot & \cdot & \cdot   & \times & \cdot & \cdot & \cdot & \cdot & \cdot &\\
& \cdot & \cdot & \cdot & \cdot & \cdot   & \cdot & \cdot & \square & \cdot & \square & \times &\\
& \cdot & \cdot & \cdot & \cdot & \cdot   & \cdot & \cdot & \times & \cdot & \cdot & \cdot &\\
&       &       &       &       &         &       &       &        &       &       &       & \ddots
\end{array}
\end{equation*}
Also, $c(5274) = (\ldots, 2,0,2,0, \ldots)$.
\end{ex}

We note without proof some simple facts about these objects. Let $e_i$ denote the vector in $\ZZ^\ZZ$ with $1$ in position $i$ and $0$ elsewhere.
\begin{prop} \label{prop:rothe-diagram-facts} For $f \in \affS_n$ and $i \in \ZZ$, it holds that
\begin{enumerate}[(a)]
\item $\#D(f) = \ell(f)$;
\item $D(f^{-1})$ is the transpose $D(f)^t$ of $D(f)$;
\item The size of row $i$ of $D(f)$ is $c_i(f)$;
\item $\ell(fs_i) < \ell(f)$ if and only if $c_i(f) > c_{i+1}(f)$, in which case $c(fs_i) = c(f)s_i - e_{i+1}$.
\end{enumerate}
\end{prop}

Let $\delta(f) \in \ZZ^\ZZ$ be the vector with $\delta_i(f) = f(i)-i$. Note that $f \in T\Bound(n)$ if and only if $\max \delta(f) - \min \delta(f) \leq n$.

\begin{lem} \label{lem:maxcode} If $f \in \affS_n^0$, then $\max \delta(f) = \max c(f)$ and $\min \delta(f) = -\max c(f^{-1})$. \end{lem}

\begin{proof}
The map $f \mapsto f^{-1}$ interchanges and negates the statistics $\max \delta(f)$ and $\min \delta(f)$, so it suffices to show that $\max \delta(f) = \max c(f)$. Write $f = w + n\lambda$ with $w \in \langle s_1, \ldots, s_n \rangle$ and $\lambda \in \ZZ^\ZZ$, as per Proposition~\ref{prop:av-properties}(b). We will prove
\begin{enumerate}[(a)]
\item $c(f) - \delta(f) = c(f^{-1})w$;
\item if $c_i(f) = \max c(f)$, then $c_{w(i)}(f^{-1}) = 0$.
\end{enumerate}
Indeed, suppose these hold. If $c_i(f) = \max c(f)$, then (b) and (a) imply $c_i(f) = \delta_i(f)$, so $\max \delta(f) \geq \max c(f)$. But by (a) the vector $c(f) - \delta(f)$ is nonnegative, so $\max \delta(f) \leq \max c(f)$.

For (a), induct on $\ell(f)$. If $\ell(f) = 0$ then both sides are the zero vector. Suppose $\ell(fs_i) < \ell(f)$. One checks $\delta(f) = \delta(fs_i)s_i + e_i - e_{i+1}$, which together with Proposition~\ref{prop:rothe-diagram-facts}(d) gives
\begin{equation} \label{eq:maxcode}
c(f) - \delta(f) = (c(fs_i) - \delta(fs_i))s_i + e_{i+1}.
\end{equation}
By induction, the right-hand side of \eqref{eq:maxcode} is $c(s_i f^{-1})w + e_{i+1}$, and the left-handed analogue of Proposition~\ref{prop:rothe-diagram-facts}(d) shows that this is equal to $c(f^{-1})w$.

For (b), suppose that $c_{w(i)}(f^{-1}) > 0$, so there exists $j > w(i)$ such that $f^{-1}(j) < f^{-1}(w(i))$; equivalently, such that $j > f(i - \lambda_i n)$ and $f^{-1}(j) < i - \lambda_i n$. Choose the $j$ with this property that maximizes $f^{-1}(j)$. Drawing a Rothe diagram makes clear the general fact that if $a < b$ and $f(a) > f(b)$, then $c_b(f) < c_a(f)$. For us, this implies that $c_i(f) = c_{i-\lambda_i n}(f) < c_{f^{-1}(j)}(f)$, so $c_i(f) \neq \max c(f)$.
\end{proof}

\begin{thm} \label{thm:bounded-diagram-characterization} Let $f \in \affS_n$. Then $f \in T \Bound(k,n)$ if and only if every row of $D(f)$ has at most $n-k$ cells and every column has at most $k$ cells. \end{thm}

\begin{proof}
Set $g = \tau^{k-\av(f)}f$. Lemma~\ref{lem:shifted-bound-k-n} implies that $f \in T \Bound(k,n)$ if and only if $g \in \Bound(n)$. By definition, $g \in \Bound(n)$ if and only if $0 \leq \min \delta(g)$ and $\max \delta(g) \leq n$. Since $\av(g) = k$, Lemma~\ref{lem:maxcode} says $\min \delta(g) = k - \max c(g^{-1})$ and $\max \delta(g) = k + \max c(g)$. But $c(g) = c(f)$, so we conclude that $f \in T \Bound(k,n)$ if and only if $\max c(f^{-1}) \leq k$ and $\max c(f) \leq n-k$, which is equivalent to the theorem by Proposition~\ref{prop:rothe-diagram-facts}.
\end{proof}

The next two lemmas will allow us to deduce facts about $\affF_f$ from information about $D(f)$.

\begin{lem}[\cite{lam-affine-stanley}, Theorem 13] \label{lem:affine-stanley-bounds} Let $\lambda_{\max}$ be the partition whose columns are the sorted column lengths of $D(f)$. If $m_{\mu}$ appears in the monomial expansion of $\tilde{F}_f$ with nonzero coefficient, then $\mu \leq \lambda_{\max}$ in dominance, and $m_{\lambda_{\max}}$ appears in $\tilde{F}_f$ with coefficient $1$. \end{lem}

Let $\omega : \Lambda \to \Lambda$ be the usual involutive ring homomorphism defined by $\omega(s_{\lambda}) = s_{\lambda^t}$, where $\lambda^t$ is the conjugate of $\lambda$. The map $\omega$ descends to a map $\Lambda^{n-k}(k) \to \Lambda^{k}(n-k)$.
\begin{lem} \label{lem:truncated-transpose} For any $0 \leq k \leq n$ and $f \in \affS_n$, we have $\omega G_{f,k} = G_{f^{-1},n-k}$. \end{lem}
\begin{proof}
Define two subspaces of $\Lambda$:
\begin{align*}
&\Lambda^{(n)} = \langle m_{\lambda} : \lambda_1 < n \rangle = \langle s_{\lambda} : \lambda_1 < n \rangle\\
&\Lambda_{(n)} = \langle h_{\lambda} : \lambda_1 < n \rangle = \langle e_{\lambda} : \lambda_1 < n \rangle.
\end{align*}
Since $m_{\lambda}$ and $h_{\lambda}$ are dual under the usual Hall inner product $\langle\, ,\,\rangle$ on $\Lambda$, the restriction of this inner product to $\Lambda_{(n)} \times \Lambda^{(n)}$ remains a perfect pairing, and so there is a unique linear involution $\omega^+ : \Lambda^{(n)} \to \Lambda^{(n)}$ defined by the property $\langle \omega(f), \omega^+(g) \rangle = \langle f, g \rangle$ where $f \in \Lambda_{(n)}$ and $g \in \Lambda^{(n)}$. Evidently $\affF_f \in \Lambda^{(n)}$, and \cite[Theorem 15 and Proposition 17]{lam-affine-stanley} prove that $\affF_{f^{-1}} = \omega^+ \affF_f$.

It suffices to see that $\omega \trunc_{k,n-k} = \trunc_{n-k,k} \omega^+$. Take $\mu \subseteq (k^{n-k})$ and $\lambda$ with $\lambda_1 < n$. The Jacobi-Trudi formula implies $s_{\mu} \in \Lambda_{(\mu_1 + \ell(\mu))} \subseteq \Lambda_{(n)}$. Then $\langle s_{\mu}, \omega^+(s_{\lambda}) \rangle = \langle s_{\mu^t}, s_{\lambda} \rangle = \delta_{\lambda, \mu^t}$.
It follows that 
\begin{equation*}
\trunc_{n-k,k} \omega^+(s_{\lambda}) = \begin{cases}
s_{\lambda^t} & \text{if $\lambda \subseteq [k] \times [n-k]$}\\
0 & \text{otherwise}
\end{cases}
\end{equation*}
as desired.
\end{proof}

\begin{thm} \label{thm:nonzero-truncation} $G_{f,k} \neq 0$ if and only if $f \in T \Bound(k,n)$. \end{thm}

\begin{proof}
If $f \in T\Bound(k,n)$, then $G_f \neq 0$ by Theorem~\ref{thm:schur-positive-truncation}. Conversely, suppose $f \notin T\Bound(k,n)$. By Theorem~\ref{thm:bounded-diagram-characterization}, there is a column of $D(f)$ with more than $k$ cells, or a row with more than $n-k$ cells. Assume for the moment that the first case holds. Letting $\lambda_{\max}$ be as in Lemma~\ref{lem:affine-stanley-bounds}, we have $\ell(\lambda_{\max}) > k$, so Lemma~\ref{lem:affine-stanley-bounds} implies
\begin{equation*}
\affF_f \in \langle m_{\mu} : \mu \leq \lambda_{\max} \rangle = \langle s_{\mu} : \mu \leq \lambda_{\max} \rangle \subseteq \langle s_{\mu} : \lambda(\mu) > k \rangle \subseteq \ker \trunc_{k,n-k}.
\end{equation*}

Now suppose $D(f)$ has a row with more than $n-k$ cells, or equivalently that $D(f^{-1})$ has a column with more than $n-k$ cells. By the previous paragraph, $G_{f^{-1},n-k} = 0$, and Lemma~\ref{lem:truncated-transpose} then implies that $G_{f,k} = 0$ as well.
\end{proof}

\section{Recurrences for affine Stanley symmetric functions} \label{sec:recurrences}

Our proof of Theorem~\ref{thm:main} will be inductive, using a recursion which arises from the following affine Chevalley formula. Given integers $i < j$ and $r$, let $c_{ij}^r$ be the number of times that $r$ occurs in $[i,j)$ modulo $n$.

\begin{thm}[\cite{affine-insertion}] \label{thm:affine-chevalley} For any $f \in \affS_n$ and $r \in \ZZ$,
\begin{equation*}
s_1 \affF_f = \sum_{\substack{1 \leq i \leq n \\ f \lessdot ft_{ij}}} c_{ij}^r \affF_{ft_{ij}}.
\end{equation*}
\end{thm}

Define sets
\begin{align*}
&\Phi^+(f,r) = \{ft_{rj} : \text{$r < j$ and $f \lessdot ft_{rj}$}\}\\
&\Phi^-(f,r) = \{ft_{ir} : \text{$i < r$ and $f \lessdot ft_{ir}$}\}.
\end{align*}
Since $c_{ij}^r - c_{ij}^{r-1}$ equals $\pm 1$ if $ft_{ij} \in \Phi^{\pm}(f,r)$ and $0$ otherwise, subtracting two instances of Theorem~\ref{thm:affine-chevalley} gives the next corollary.
\begin{cor} \label{cor:affine-transitions} For any $f \in \affS_n$ and $r \in \ZZ$,
\begin{equation*}
\sum_{g \in \Phi^-(f,r)} \affF_g = \sum_{g \in \Phi^+(f,r)} \affF_g.
\end{equation*}
\end{cor}
We call Corollary~\ref{cor:affine-transitions} the \emph{(affine) transition formula}, since it is an analogue of the transition formula of Lascoux and Sch\"utzenberger \cite{lascouxschutzenbergertree} for Schubert polynomials and ordinary Stanley symmetric functions; for a combinatorial proof of Corollary~\ref{cor:affine-transitions}, see \cite{affine-little-bump}.


Next we record truncated versions of the preceding two identities. For $r \in \ZZ$, define sets
\begin{gather*}
B\Phi^{\pm}(f,r) = \Phi^{\pm}(f,r) \cap \Bound(k,n)\\
\BCov_r(f) = \{(i,j) \in [n] \times \ZZ : \text{$c_{ij}^r \neq 0$ and $f \lessdot ft_{ij} \in \Bound(n)$} \}.
\end{gather*}
\begin{lem} \label{lem:bounded-cover} If $f \in \Bound(n)$ and $(i,j) \in \BCov_r(f)$ for some $r$, then $j-i < n$. \end{lem}
\begin{proof}
If $j-i > n$, then $(ft_{ij})(i) = f(j) \geq j > i+n$, so $ft_{ij}$ is not bounded. No inversion $(i,j)$ of an affine permutation can have $j \equiv i \pmod{n}$, so we conclude $j-i < n$.
\end{proof}

\begin{prop} \label{prop:recurrences} For $f \in \Bound(k,n)$ and any $r \in \ZZ$, the following two identities hold in $\Lambda^{n-k}(k)$:
\begin{gather}
s_1 G_f = \sum_{(i,j) \in \BCov_r(f)} G_{ft_{ij}} \tag{a}\\
\sum_{g \in B\Phi^-(f,r)} G_{g} = \sum_{g \in B\Phi^+(f,r)} G_{g} \tag{b}.
\end{gather}
\end{prop}

\begin{proof}
Part (b) follows from part (a) just as Corollary~\ref{cor:affine-transitions} follows from Theorem~\ref{thm:affine-chevalley}. As for part (a), Theorems~\ref{thm:affine-chevalley} and \ref{thm:nonzero-truncation} show that
\begin{equation*}
s_1 G_f = \sum_{(i,j)} c^r_{ij} G_{ft_{ij}}
\end{equation*}
where $(i,j)$ runs over $i < j$ such that $f \lessdot ft_{ij}$ and $ft_{ij} \in T\Bound(k,n)$. Since $\av(f) = k$, Lemma~\ref{lem:shifted-bound-k-n} shows that in fact $ft_{ij} \in T\Bound(k,n)$ is equivalent to $ft_{ij} \in \Bound(k,n)$, so we can assume $(i,j) \in \BCov_r(f)$. Now Lemma~\ref{lem:bounded-cover} implies $c^{ij}_r = 1$.
\end{proof}

The \emph{maximal inversion} of $f \in \Bound(k,n)$ is the lexicographically maximal pair $(r,s)$ with $1 \leq r < s \leq n$ and $f(r) > f(s)$. If no such pair exists, we say $f$ is \emph{$0$-Grassmannian}. Equivalently, the descent set $\Des(f) = \{i \in \ZZ : f(i) > f(i+1)\}$ is contained in $0+n\ZZ$.

\begin{defn}
The \emph{bounded affine Lascoux-Sch\"utzenberger (L-S) tree} of $f$ is a rooted tree whose vertices are labeled by elements of $\Bound(k,n)$ and whose edges are labeled by $+$ or $-$, defined as follows:
\begin{itemize}
\item The root of the tree is $f$;
\item If a vertex $g$ is $0$-Grassmannian, then $g$ has no children;
\item If $g$ is not $0$-Grassmannian, then $g$ has children $B\Phi^{-}(ft_{rs}, r)$ (via edges labeled ${+}$) and $B\Phi^{+}(ft_{rs},r) \setminus \{f\}$ (via edges labeled ${-}$).
\end{itemize}
\end{defn}
The maximality of $(r,s)$ implies that $f \lessdot ft_{rs}$, and hence $f \in B\Phi^+(ft_{rs}, r)$. Applying Proposition~\ref{prop:recurrences} to $ft_{rs}$ and solving for $G_f$ on the right-hand side yields the next proposition.

\begin{prop} \label{prop:LS-recurrence}
If $g$ is any non-leaf vertex of a bounded affine L-S tree, then $G_g = \sum_{h^+} G_{h^+} - \sum_{h^-} G_{h^-}$ where $h^+$ and $h^-$ run over the children of $g$ connected by edges labeled ${+}$ and ${-}$ respectively.
\end{prop}

\begin{thm} \label{thm:finite-LS-tree}
The bounded affine L-S tree of any $f \in \Bound(k,n)$ is finite.
\end{thm}

\begin{proof}
Let $w_f$ denote the word $f(1)\cdots f(n)$, and write $\ell(w_f)$ for the number of inversions in $w_f$ (which may be less than $\ell(f)$). Write $<_{\lex}$ for lexicographic order on $\ZZ^n$. We claim that if $g$ is a child of $f$, then either $\ell(w_g) < \ell(w_f)$, or $\ell(w_g) = \ell(w_f)$ and $g >_{\lex} f$. Since $\Bound(k,n)$ is finite, this will imply that every sufficiently long path from the root encounters a vertex $h$ with $\ell(w_h) = 0$, which by definition is a leaf.

To prove the claim, first consider two cases:
\begin{enumerate}[(i)]
\item $g = ft_{rs} t_{jr} \in B\Phi^{-}(ft_{rs})$ and $1 \leq j < r$: Here $w_g$ is obtained from $w_f$ by replacing the subsequence $f(j), f(r), f(s)$ with $f(s), f(j), f(r)$. Since $j, r, s \in [n]$, the Bruhat cover relations $ft_{rs} \lessdot f, g$ imply that $\ell(w_g) = \ell(w_f)$. They also imply $f(j) < f(s) < f(r)$, so $w_g >_{\lex} w_f$.

\item $g = ft_{rs} t_{jr} \in B\Phi^{-}(ft_{rs})$ and $j < 1$: Let us see that $\ell(w_g) = \ell(w_f) - 1$. Suppose $p < q$, $f(p) > f(q)$ is an inversion of $f$; we may assume $q \in [n]$. Boundedness implies $1 \leq q \leq f(q) < f(p) \leq p+n$, hence $p \geq -n+1$. Thus, $\ell(f)$ counts the inversions in the word $f(-n+1)\cdots f(0)f(1)\cdots f(n)$ which end at a position $q \geq 1$, while $\ell(w_f)$ counts the subset starting at a position $p \geq 1$. In passing from $f$ to $g$, we lose the inversion in positions $1 \leq r < s \leq n$, gain an inversion in positions $j < 1 \leq s$, and preserve all other inversions, so we conclude that $\ell(w_g) = \ell(w_f)-1$. 
\end{enumerate}
The arguments for $g = ft_{rs}t_{rj} \in B\Phi^{-}(ft_{rs})$ are similar, with the cases $j < s$ and $j > s$ corresponding to (i) and (ii) respectively.
\end{proof}

If $f \in \Bound(k,n)$ is $0$-Grassmannian, then $f = a_1 \cdots a_{n-k} b_1 \cdots b_{k}$ where
\begin{equation*}
1 \leq a_1 < \cdots < a_{n-k} \leq n < b_1 < \cdots < b_k \leq 2n
\end{equation*}
and $\{b_1 - n, \ldots, b_k - n\} = [n] \setminus \{a_1, \ldots, a_k\}$. Thus, $f\tau^{-k}$ is in $S_n$ and has descent set $k + n\ZZ$; that is, it is a $k$-Grassmannian permutation. Recall that the \emph{shape} of a $k$-Grassmannian permutation $w$ is the partition obtained by sorting $c(w)$, or whose Young diagram is obtained by deleting empty rows and columns in $D(w)$. Define the shape $\lambda(f)$ of a $0$-Grassmannian $f \in \Bound(k,n)$ to be the shape of $f\tau^{-k}$.

\begin{lem} \label{lem:schur-leaves}
If $f \in \Bound(k,n)$ is $0$-Grassmannian, then $G_f = \bar s_{\lambda(f)}$.
\end{lem}

\begin{proof}
It is well-known that if $w \in S_n$ is Grassmannian of shape $\lambda$, then the ordinary Stanley symmetric function $F_w = \affF_w$ is the Schur function $s_{\lambda}$ (by \cite[Theorem 4.1]{stanleysymm}, for instance). Hence $\affF_f = \affF_w = s_{\lambda}$.
\end{proof}

If $\pi$ is a path in a graph with edges labeled by $\pm$, let $\sgn(\pi)$ be the product of the edge labels.
\begin{thm} \label{thm:LS-formula}
Given a vertex $g$ of the bounded affine L-S tree of $f \in \Bound(k,n)$, let $\pi_g$ denote the unique path from the root $f$ to $g$. Then
\begin{equation*}
G_f = \sum_g \sgn(\pi_g) s_{\lambda(g)}
\end{equation*}
where $g$ runs over the leaves of the tree.
\end{thm}

\begin{rem} We are abusing notation slightly by conflating vertices of the L-S tree with their labels: it is possible for different vertices to have the same label. \end{rem}

\begin{proof}
The sum makes sense by Theorem~\ref{thm:finite-LS-tree}, and applying induction to Proposition~\ref{prop:LS-recurrence} with Lemma~\ref{lem:schur-leaves} gives the result.
\end{proof}

When $f \in S_n$, the bounded affine L-S tree reduces to the transition tree of \cite{lascouxschutzenbergertree}. Indeed, in this case $\Phi^+(ft_{rs}, r) = \{f\}$, so the L-S tree has no edges labeled ${-}$, and Theorem~\ref{thm:LS-formula} exhibits $G_f$ (and in fact $\affF_f$) as Schur-positive. In the general affine case the Schur-positivity of $G_f$ is not clear from Theorem~\ref{thm:LS-formula}, but that recurrence is still a much more effective means of computing $G_f$ than the definition in terms of cyclically decreasing factorizations. The truncation is essential here: it is not clear if there is any analogue of maximal inversion which can be used to construct a usefully terminating recursion for $\affF_f$ from Corollary~\ref{cor:affine-transitions}.

\section{Generalized Schur modules} \label{sec:schur}
Let $S_D$ denote the group of permutations of a finite set $D$. Given a set partition $\pi$ of $D$, let $\stb \pi \subseteq S_D$ be the stabilizer of $\pi$ under the action of $S_D$ on all partitions of $D$. The \emph{Young symmetrizer} associated to a pair of partitions $\pi_R, \pi_C$ of $D$ is an element of the group algebra $\CC[S_D]$:
\begin{equation*}
y_{\pi_R, \pi_C} = \sum_{\substack{p \in \stb \pi_R \\ q \in \stb \pi_C}} \sgn(q)qp.
\end{equation*}

Fix a finite-dimensional complex vector space $V$. Say $D$ is a set with $d$ elements. Let $V^{\otimes D}$ denote the $d$-fold tensor product of $V$ with itself, where we think of the tensor factors as labeled by elements of $D$ rather than $[d]$ (alternatively, define $V^{\otimes D}$ to be the space of multilinear functions $(V^*)^D \to \CC$, where $(V^*)^D$ is the space of functions $D \to V^*$). The space $V^{\otimes D}$ is naturally a right $S_D$-module and a left $\GL(V)$-module, and these two actions commute. 
\begin{defn} The \emph{(generalized) Schur module} for $\GL(V)$ associated to a pair of partitions $\pi_R, \pi_C$ of a finite set $D$ is the left $\GL(V)$-module $V[\pi_R, \pi_C] := V^{\otimes D} y_{\pi_R, \pi_C}$. \end{defn}

A \emph{diagram} is a finite subset of $\ZZ^2$, or more generally of $\ZZ^2$ modulo some equivalence relation. It is common to label generalized Schur modules by diagrams, as follows. The $i$th row of a diagram $D$ is $\{(i,j) \in D : j \in \ZZ\}$, and the \emph{row partition} of $D$ is
\begin{equation*}
\pi_R = \{\{(i,j) \in D : j \in \ZZ\} : i \in \ZZ\}.
\end{equation*}
Define the columns and column partition $\pi_C$ of $D$ analogously. The Schur module of $D$ is now $V[D] := V[\pi_R, \pi_C]$. This definition is (essentially) due to James and Peel in \cite{jamespeel}. The next lemma shows that every nonzero Schur module arises from a diagram.
\begin{lem} \label{lem:nonzero-schur} Suppose $\pi_R$ and $\pi_C$ are partitions of $D$. Define a function $\iota : D \to \pi_R \times \pi_C$ by letting $\iota(x)$ be the unique pair of blocks $(b,b')$ such that $x \in b \cap b'$. Observe that $\iota$ is injective if and only if every pair $b \in \pi_R$, $b' \in \pi_C$ satisfies $\#(b \cap b') \leq 1$.
\begin{enumerate}[(a)]
\item If $\iota$ is not injective, then $V[\pi_R, \pi_C] = 0$.
\item Suppose $\iota$ is injective. Numbering the blocks of $\pi_R$ and $\pi_C$ identifies $\iota(D) \subseteq \pi_R \times \pi_C$ with a diagram in $\ZZ^2$, and $V[\pi_R, \pi_C] = V[\iota(D)]$.
\end{enumerate}
\end{lem}

\begin{proof}
\begin{enumerate}[(a)]
\item Say $\iota(x) = \iota(y)$ and $x \neq y$. Then $\stb(\pi_R) \cap \stb(\pi_C)$ contains the transposition $t := (x\,y)$. The equations 
\begin{equation*}
\sum_{q \in \stb(\pi_C)} \sgn(q)tq = \sum_{q \in \stb(\pi_C)} \sgn(q)qt = -\sum_{q \in \stb(\pi_C)} \sgn(q)q \quad \text{and} \quad \sum_{p \in \stb(\pi_R)} tp = \sum_{p \in \stb(\pi_R)} p
\end{equation*}
imply that $y_{\pi_R, \pi_C} = ty_{\pi_R, \pi_C} = -y_{\pi_R, \pi_C}$, so $y_{\pi_R, \pi_C} = 0$.

\item If $\iota$ is injective, the pair of partitions associated to the diagram $\iota(D)$ is $\iota(\pi_R), \iota(\pi_C)$, so $V[\iota(D)] = V[\pi_R, \pi_C]$.
\end{enumerate}
\end{proof}

With Lemma~\ref{lem:nonzero-schur} in mind, we will now index Schur modules by diagrams. Write $R(D)$ and $C(D)$ for the subgroups of $S_D$ stabilizing the associated partitions $\pi_R$ and $\pi_C$, and $y_D$ for the corresponding Young symmetrizer. Although Lemma~\ref{lem:nonzero-schur} shows that one can always use ordinary diagrams in $\ZZ^2$ for this purpose, for us it will frequently be more natural to consider \emph{cylindric diagrams} instead, i.e. finite subsets of some cylinder $\Cyl_{k,m}$. In this language, Lemma~\ref{lem:nonzero-schur} takes the following form. Let $\rho : \Cyl_{k,m} \to \ZZ^2/(k\ZZ \times m\ZZ)$ denote the quotient map from the $(k,m)$-cylinder to the $(k,m)$-torus. Call a cylindric diagram $D \subseteq \Cyl_{k,m}$ \emph{toric} if $\rho$ is injective on $D$.

\begin{lem} \label{lem:toric} Let $D$ be a cylindric diagram.
\begin{enumerate}[(a)]
\item If $D$ is not toric, then $V[D] = 0$.
\item If $D$ is toric, then $V[D] \simeq V[\rho(D)]$.
\end{enumerate}
\end{lem}

\begin{proof}
Observe that $D$ is toric if and only if a row and column of $D$ never intersect in more than one cell, and apply Lemma~\ref{lem:nonzero-schur}.
\end{proof}

\begin{defn} \label{defn:meets}
Suppose $D$ is toric. If $(i + k\ZZ) \times (j + m\ZZ)$ intersects $D$, write $(i,j)_D$ for the unique point of intersection; if not, we will say $D$ does not contain $(i,j)_D$. More generally, for a subset $A \subseteq \Cyl_{m,n}$ we will say that $A$ \emph{meets} $D$ if $A + m\ZZ \times n\ZZ$ intersects $D$. For a single point $p$, we say $D$ meets $p$ to mean $D$ meets $\{p\}$.
\end{defn}

Classically, Schur modules for $\GL(V)$ are the modules $V[D]$ where $D$ runs over Young diagrams of partitions $\lambda$ with $\ell(\lambda) \leq \dim V$. As per Theorem~\ref{thm:rep-theory-facts}, it is a basic fact that (over $\CC$) these form a complete and irredundant set of irreducible polynomial representations of $\GL(V)$ and that $\ch V[\lambda]$ is the Schur polynomial $s_{\lambda}(x_1, \ldots, x_k)$. When one allows arbitrary diagrams $D$, the modules $V[D]$ need not be irreducible any longer, and it is an interesting open problem to find combinatorial descriptions of the multiplicities of irreducible submodules of $V[D]$ for general $D$.

Upon fixing a basis $e_1, \ldots, e_k$ of $V$, fillings of $D$ by $[k]$ naturally index a basis of $V^{\otimes D}$. That is, to each filling $T : D \to [k]$ we associate the pure tensor $e_T := \bigotimes_{x \in D} e_{T(x)}$, where the tensor product is taken in some fixed order.
\begin{ex} \label{ex:schur}
Let
\begin{equation*}
D = \begin{array}{ccc}
\square & \cdot & \square\\
\cdot   & \square & \cdot
\end{array} = \{(1,1), (2,2), (1,3)\}.
\end{equation*}
Let $\square = (1,1)$ and $\square' = (1,3)$. Then $C(D)$ is trivial, while $R(D) = \{1, (\square\,\, \square')\}$, so $y_D = 1 + (\square\,\,\square')$.

In this example, take $\dim V = 2$, and define
\begin{equation*}
\alpha = \begin{array}{ccc}
1 & \cdot & 1\\
\cdot & 1 & \cdot
\end{array}
\qquad
\beta = \begin{array}{ccc}
2 & \cdot & 2\\
\cdot & 2 & \cdot
\end{array} \qquad
\gamma = \begin{array}{ccc}
1 & \cdot & 1\\
\cdot & 2 & \cdot
\end{array} \qquad
\delta = \begin{array}{ccc}
2 & \cdot & 2\\
\cdot & 1 & \cdot
\end{array}
\end{equation*}
\begin{equation*}
\epsilon = \begin{array}{ccc}
1 & \cdot & 2\\
\cdot & 1 & \cdot
\end{array} + \begin{array}{ccc}
2 & \cdot & 1\\
\cdot & 1 & \cdot
\end{array} \qquad
\zeta = \begin{array}{ccc}
1 & \cdot & 2\\
\cdot & 2 & \cdot
\end{array} + \begin{array}{ccc}
2 & \cdot & 1\\
\cdot & 2 & \cdot
\end{array}
\end{equation*}
Here we are identifying fillings $T$ with the corresponding tensors $e_T \in V^{\otimes D}$. It is clear that $\{\alpha, \beta, \gamma, \delta, \epsilon, \zeta\}$ form a basis for $V[D]$, so $\ch V[D] = x_1^3 + x_2^3 + 2x_1^2 x_2 + 2x_1 x_2^2 = s_3 + s_{21}$, and we conclude $V[D] \simeq V[3] \oplus V[21]$. More explicitly, $\alpha, \beta, \gamma+\epsilon, \delta+\zeta$ span the submodule of $V[D]$ consisting of symmetric tensors, which is isomorphic to $V[3] \simeq \operatorname{Sym}^3(\CC^2)$. One can check that in the direct sum decomposition,
\begin{equation*}
V[D] = \langle \alpha, \beta, \gamma+\epsilon, \delta+\zeta \rangle \oplus \langle \epsilon - 2\delta, \zeta - 2\gamma \rangle,
\end{equation*}
the second summand is $\GL(V)$-invariant and isomorphic to the other simple factor $V[21]$.
\end{ex}

Take two ordinary diagrams $D, D' \subseteq \ZZ^2$. We say $D$ and $D'$ are \emph{equivalent} if they are conjugate under the left action of $S_\ZZ \times S_\ZZ$ on $\ZZ^2$. Define a new diagram $D \ast D' = D \cup \{(i+N, j+N) : (i,j) \in D'\}$, where $N$ is sufficiently large that this union is disjoint. Up to equivalence, $D \ast D'$ does not depend on the choice of $N$. One can define equivalence for toric diagrams by replacing $S_\ZZ \times S_\ZZ$ with $\affS_k \times \affS_m$. 

\begin{prop} \label{prop:diagram-facts}  \hfill
\begin{enumerate}[(a)]
\item If $D$ and $D'$ are equivalent, then $V[D] \simeq V[D']$.
\item $V[D \ast D'] \simeq V[D] \otimes V[D']$, and $\ch V[D \ast D'] = (\ch V[D])(\ch V[D'])$.
\end{enumerate}
\end{prop}

\begin{proof}\hfill
\begin{enumerate}[(a)]
\item The cells of equivalent diagrams are in bijection, and this bijection respects the relevant partitions of cells into rows and columns.
\item Identifying $D'$ with the corresponding subdiagram of $D \ast D'$, we have $y_{D \ast D'} = y_D y_{D'}$, and then
\begin{align*}
V[D \ast D'] &= V^{\otimes D \ast D'}y_{D \ast D'} = V^{\otimes D \ast D'}y_{D} y_{D'} \simeq (V^{\otimes D} \otimes V^{\otimes D'}) y_D y_{D'}\\
&= V^{\otimes D}y_D \otimes V^{\otimes D'}y_{D'} = V[D] \otimes V[D'].
\end{align*}
\end{enumerate}
\end{proof}

The next lemma is a key tool for decomposing Schur modules of diagrams. Given two ordered pairs $x = (i,j)$ and $x' = (i',j')$, write $x|x' = (i,j')$. Given a cylindric diagram $D$ and ordered pairs $x, x'$, define a function $R_x^{x'} : D \to \Cyl_{m,n}$ as follows:
\begin{itemize}
\item If $p \not\equiv i \pmod{m}$ then $R_x^{x'}$ fixes $(p,q)_D$;
\item If $D$ contains $(i,q)_D$ and $(i',q)_D$, then $R_x^{x'}$ fixes $(i,q)_D$;
\item If $(i,q)_D = (i,q+pn)$ and $D$ does not contain $(i',q)_D$, then $R_x^{x'}(i,q)_D = (i',q+pn)$.
\end{itemize}
Now let $R_x^{x'} D$ be the image of $D$ under $R_x^{x'}$. Define $C_x^{x'} D$ by modifying columns $j$ and $j'$ analogously. We call the operators $R_x^{x'}$ and $C_x^{x'}$ \emph{James-Peel moves}.  We make the same definitions for ordinary diagrams, taking $(i,j)_D$ to mean $(i,j)$. The next proposition is clear from the definitions.

\begin{prop} \label{prop:jp-toric}
If $D$ is toric, then so are $R_{x}^{x'}D$ and $C_{x}^{x'}D$, and these James-Peel moves commute with the quotient map $\rho$.
\end{prop}

\begin{rem} \label{rem:toric-james-peel}
For simplicity, the results involving James-Peel moves that follow have been phrased for ordinary diagrams. However, Lemma~\ref{lem:toric} and Proposition~\ref{prop:jp-toric} show that one could just as well state them for toric diagrams by replacing ``intersects'' with ``meets'' (see Definition~\ref{defn:meets}), and we will take this point of view in Section~\ref{sec:main}.
\end{rem}

\begin{lem} \label{lem:jp}
Let $D$ be a diagram, and $x, x'$ any two points in $\ZZ$. Then $V[C_{x}^{x'}D] \subseteq V[D]$, and there exists a surjective homomorphism $\phi_x^{x'} : V[D] \twoheadrightarrow V[R_{x}^{x'} D]$. If moreover $x, x' \in D$ but $x|x', x'|x \notin D$, then $V[C_{x}^{x'}D] \subseteq \ker \phi_x^{x'}$; in particular, over $\CC$ one has 
\begin{equation*}
V[R_{x}^{x'} D] \oplus V[C_{x}^{x'}D] \hookrightarrow V[D].
\end{equation*}
\end{lem}
\begin{proof}
The same statement for the \emph{Specht modules} $\CC[S_D]y_D$ appears as \cite[Theorem 2.4]{jamespeel}; the proof given there operates entirely on the level of the Young symmetrizers $y_D$, so it applies just as well to Schur modules.
\end{proof}

\begin{ex}
Let
\begin{equation*}
D = \begin{array}{cccc}
\square & \square & \cdot & \cdot\\
\square & \cdot & \square & \cdot
\end{array} = \{(1,1), (1,2), (2,1), (2,3)\}.
\end{equation*}
Then
\begin{equation*}
R_{(1,2)}^{(2,3)} D = 
\begin{array}{cccc}
\square & \cdot & \cdot & \cdot\\
\square & \square & \square & \cdot
\end{array} \quad \text{and} \quad
C_{(1,2)}^{(2,3)} D = \begin{array}{cccc}
\square & \cdot & \square & \cdot\\
\square & \cdot & \square & \cdot
\end{array}
\end{equation*}
Since $(1,2), (2,3) \in D$ but $(1,3), (2,2) \notin D$, Lemma~\ref{lem:jp} and Proposition~\ref{prop:diagram-facts}(a) imply $V[31] \oplus V[22] \hookrightarrow V[D]$. In fact, this injection is an isomorphism: $D$ is equivalent to the skew shape $(3,2) \setminus (1)$ and $s_{32 \setminus 1} = s_{31} + s_{22}$.
\end{ex}

We will not need to know an exact definition for the surjections in Lemma~\ref{lem:jp}, but we will need to know how they restrict to submodules of the form $V[C_{y}^{y'}D]$. 
\begin{lem}[\cite{billey-pawlowski-2013}, Lemma 3.3] \label{lem:jp-commute}
Let $D$ be a diagram and $x, x', y, y' \in \ZZ^2$. Consider the two surjections
\begin{align*}
&\phi_x^{x'} : V[D] \twoheadrightarrow V[R_x^{x'}D]\\
&\widetilde{\phi}_x^{x'} : V[C_y^{y'}D] \twoheadrightarrow V[R_x^{x'} C_y^{y'}D]
\end{align*}
given by Lemma~\ref{lem:jp}. 
\begin{enumerate}[(a)]
\item $R_x^{x'} C_y^{y'} D \neq C_y^{y'} R_x^{x'} D$ if and only if $x|y, x'|y' \in D$ but not $x|y', x'|y \notin D$, or $x|y', x'|y \in D$ but $x|y, x'|y' \notin$;
\item If $R_x^{x'} C_y^{y'} D = C_y^{y'} R_x^{x'} D$, then $\widetilde{\phi}_x^{x'}$ is the restriction of $\phi_x^{x'}$ to $V[C_y^{y'}D] \subseteq V[D]$, so $\phi_x^{x'} V[C_y^{y'}D] = V[R_x^{x'} C_y^{y'}D]$;
\item If $R_x^{x'} C_y^{y'} D \neq C_y^{y'} R_x^{x'} D$, then $\phi_x^{x'} V[C_y^{y'}D] = 0$.
\end{enumerate}
\end{lem}
Pictorially, the condition in Lemma~\ref{lem:jp-commute}(a) says that the intersection of $D$ with the two rows of $x, x'$ and the two columns of $y, y'$ is either $\begin{array}{cc} \times & \cdot \\ \cdot & \times \end{array}$ or $\begin{array}{cc} \cdot & \times \\ \times & \cdot \end{array}$. Write $D_{x \to y}$ for $R_x^y C_x^y D$; we will only use this notation when $R_x^y C_x^y D = C_x^y R_x^y D$.

The next definition and lemma can be thought of as a generalization of Pieri's rule for computing the Schur expansion of $s_1 s_{\lambda}$. A subset $\Delta \subseteq \ZZ^2$ is a \emph{transversal} if no two of its points are in the same row or column. For two sets $X, Y \subseteq \ZZ^2$, define $X|Y = \{x|y : x \in X, y \in Y\}$.
\begin{defn} \label{defn:corner}
A \emph{corner configuration} for a diagram $D$ is a pair $(a, \Delta)$ where $\Delta \subseteq \ZZ^2$ is a totally ordered set, $a \in \ZZ^2$, and:
\begin{enumerate}[(a)]
\item $\{a\} \cup \Delta$ is a transversal disjoint from $D$;
\item $\{a\} | \Delta$ and $\Delta | \{a\}$ are disjoint from $D$;
\item If $x < y$ are in $\Delta$, then $y|x \in D$.
\end{enumerate}
\end{defn}

Up to taking duals, the next lemma appears as \cite[Proposition 3.2]{liubranching}.
\begin{lem} \label{lem:corner-configurations} If $(a, \Delta)$ is a corner configuration for $D$, then
\begin{equation*}
\bigoplus_{x \in \Delta} V[(D \cup \{a\})_{a \to x}] \hookrightarrow V[\square \ast D].
\end{equation*}
\end{lem}

\begin{ex}
Recall that the \emph{outside corners} of a Young diagram $D = \lambda$ are the points which can be added to $\lambda$ to obtain another Young diagram. If $\Delta$ is the set of outside corners ordered from top to bottom, and $a$ is any point not in the row or column of a cell of $\lambda$, then $(a, \Delta)$ is a corner configuration for $\lambda$. Lemma~\ref{lem:corner-configurations} then recovers Pieri's rule for $s_1 s_\lambda$, except that one does not know a priori that the injection is an isomorphism. In general, Lemma~\ref{lem:corner-configurations} allows $a$ to share a row or column with a cell of $D$, in which case $(D \cup \{a\})_{a \to x}$ may differ from $D \cup \{x\}$.
\end{ex}

Unfortunately, Lemma~\ref{lem:corner-configurations} is too weak for our purposes. Our goal is to show that $\trch V[D(f)]$ satisfies the same bounded affine L-S recurrence (Proposition~\ref{prop:LS-recurrence}) as $G_f$. When $f \in S_n$, there are no minus signs on the right-hand side of Proposition~\ref{prop:LS-recurrence}, and one can indeed obtain the desired Schur module analogue by applying Lemma~\ref{lem:corner-configurations} to an appropriate corner configuration \cite{billey-pawlowski-2013}. When there are minus signs in the bounded affine L-S tree this would make no sense, so instead we prove a Schur module version of Proposition~\ref{prop:recurrences}(a). However, the sum in Proposition~\ref{prop:recurrences}(a) has many more terms than the ones in Proposition~\ref{prop:LS-recurrence}, and they can be thought of as arising from multiple, interacting corner configurations, as in the next definition.

\begin{defn} \label{defn:corner-system} A \emph{system of corner configurations} for a diagram $D$ is a finite totally ordered set $\Cor$ of corner configurations such that $\{a : (a, \Delta) \in \Cor\}$ is a transversal, and for all $(a, \Delta_a), (b, \Delta_b) \in \Cor$,
\begin{enumerate}[(a)]
\item $\{a\} \cup \Delta_b$ is a transversal;
\item If $(a, \Delta_a) < (b, \Delta_b)$ and $x \in \Delta_a, y \in \Delta_b$, then $\{b, y\} | \{a, x\}$ intersects $D$.
\end{enumerate}
\end{defn}

\begin{lem} \label{lem:corner-system} Let $\Cor$ be a system of corner configurations for a diagram $D$. Then
\begin{equation*}
\bigoplus_{\substack{(a, \Delta) \in \Cor \\ x \in \Delta}} V[(D \cup a)_{a \to x}] \hookrightarrow V[\square \ast D].
\end{equation*}
\end{lem}

\begin{proof}
Let $A = \{a : (a, \Delta) \in \Cor\}$ and for $a \in A$, let $\Delta_a$ be such that $(a, \Delta_a) \in \Cor$. The set $A$ inherits a total order from the total order on $\Delta$. Write $<$ for the lexicographic ordering on pairs $(a, x)$ with $a \in \Delta_a$ and $x \in \Delta_a$, with $\lessdot$ being its covering relation. We will construct a filtration of $V[D \ast \sq]$ by modules $N_b^y$ for $b \in A$ and $y \in \Delta_b$, ordered by $\lessdot$, such that if $(b',y') \lessdot (b,y)$, then there is a surjective homomorphism $N_b^{y} / N_{b'}^{y'} \twoheadrightarrow V[(D \cup b)_{b \to y}]$. Since we are working over $\CC$, this will prove the lemma. Write $E = D \ast \sq$, and define
\begin{equation*}
N_b^y = \sum_{(a,x) \leq (b,y)} V[C_{a}^{x} C_{\sq}^{a} E].
\end{equation*}
By Lemma~\ref{lem:jp}, this is a submodule of $V[E]$.

Let $\phi_{\sq}^b : V[E] \twoheadrightarrow V[R_{\sq}^b E]$ and $\phi_b^y : V[R_{\sq}^b E] \twoheadrightarrow V[R_b^y R_{\sq}^b E]$ be the surjections given by Lemma~\ref{lem:jp}. We now show that
\begin{equation*}
\phi_b^y \phi_{\sq}^b V[C_{a}^{x} C_{\sq}^{a} E] = \begin{cases}
V[(D \cup b)_{b \to y}] & \text{if $(a,x) = (b,y)$}\\
0 & \text{if $(a,x) < (b,y)$}
\end{cases}.
\end{equation*}
This will prove the lemma, since it shows that $\phi_{\sq}^b \phi_b^y$ descends to a surjection $N^y_b / N^{y'}_{b'} \twoheadrightarrow V[(D \cup b)_{b \to y}]$ where $(b', y') \lessdot (b,y)$.

The key point is that by Lemma~\ref{lem:jp-commute}, either:
\begin{enumerate}[(i)]
\item both of the $R$ operators in the expression $R_b^y R_\sq^b C_{a}^{x} C_{\sq}^{a} E$ commute past both of the $C$ operators and $\phi_b^y \phi_{\sq}^b V[C_{a}^{x} C_{\sq}^{a} E] = V[R_b^y R_\sq^b C_{a}^{x} C_{\sq}^{a} E]$, or;
\item one of the $R$'s fails to commute with one of the $C$'s and $\phi_{\sq}^b \phi_b^y V[C_{a}^{x} C_{\sq}^{a} E] = 0$.
\end{enumerate}
What we need to see is that (i) holds if $(a,x) = (b,y)$, while (ii) holds if $(a,x) < (b,y)$.

First consider the case $(a,x) = (b,y)$. According to Lemma~\ref{lem:jp-commute}, $R_\sq^b C_{b}^{y} C_{\sq}^{b} E \neq C_{b}^{y} R_\sq^b  C_{\sq}^{b} E$ if and only if $\sq|b, b|y \in C_{\sq}^{b} E$ and $\sq|y, b|b \notin C_{\sq}^b E$, or $\sq|b, b|y \notin C_{\sq}^{b} E$ and $\sq|y, b|b \in C_{\sq}^b E$. Since $C_{\sq}^{b} E$ contains $\sq|b$ but not $b|y$ (the latter by Definition~\ref{defn:corner}(b)), this condition does not hold. Thus $R_\sq^b C_{b}^{y} C_{\sq}^{b} E = C_{b}^{y} R_\sq^b  C_{\sq}^{b} E$. The remaining arguments are similar:
\begin{itemize}
\item To conclude $R_\sq^b  C_{\sq}^{b} E =  C_{\sq}^{b} R_\sq^b E$, use that $E$ contains $\sq$ but not $b$ (because $A$ is disjoint from $D$).
\item To conclude $R_b^y C_{b}^{y} C_{\sq}^{b} R_\sq^b E = C_{b}^{y} R_b^y C_{\sq}^{b} R_\sq^b E$, use that $C_{\sq}^{b} R_\sq^b E$ contains $b$ but not $y$ (because $\Delta_b$ is disjoint from $D$).
\item To conclude $R_b^y C_{\sq}^{b} R_\sq^b E = C_{\sq}^{b} R_b^y R_\sq^b E$, use that $R_\sq^b E$ contains $b|\sq$ but not $y|b$ (by Definition~\ref{defn:corner}(b)).
\end{itemize}
Therefore case (i) holds, meaning
\begin{equation*}
\phi_b^y \phi_{\sq}^b V[C_{a}^{x} C_{\sq}^{a} E] = V[R_b^y R_\sq^b C_{b}^{y} C_{\sq}^{b} E] = V[(D \cup b)_{b \to y}]
\end{equation*}
as desired.

Now suppose $(a,x) < (b,y)$. Our goal is to see that case (ii) above holds.
\begin{enumerate}[(1)]
\item Certainly $\sq|a \in C_\sq^a E$, while $\sq|x \notin C_\sq^a E$ because $\{a\} \cup \Delta_a$ is a transversal. Thus if $b|x \in D$ and $b|a \notin D$, we can conclude that $R_\sq^b C_{a}^{x} C_{\sq}^{a} E \neq C_{a}^{x} R_\sq^b  C_{\sq}^{a} E$, so that we are in case (ii). Otherwise, if $b|a \in D$, proceed to (2), while if $b|a \notin D$ and $b|x \notin D$, proceed to (3).

\item We can now assume that $R_\sq^b C_{a}^{x} C_{\sq}^{a} E = C_{a}^{x} R_\sq^b  C_{\sq}^{a} E$ and $b|a \in D$. Since $E$ contains $\sq|\sq$ and $b|a$ but not $\sq|a$ or $b|\sq$, we conclude that $R_\sq^b  C_{\sq}^{a} E \neq C_{\sq}^{a} R_\sq^b E$, so we are in case (ii).

\item We can now assume that
\begin{equation*}
R_\sq^b C_{a}^{x} C_{\sq}^{a} E = C_{a}^{x} R_\sq^b  C_{\sq}^{a} E = C_a^x C_{\sq}^a R_\sq^b E.
\end{equation*}
and that $b|a, b|x \notin D$. Suppose for the moment that $y|a \notin D$ and $y|x \in D$. Then $R_\sq^b C_\sq^a E$ contains $b|a$ and $y|x$, but not $y|a$ or $b|x$ (because $y|a, b|x \notin D$ and $\{a, y\}$ and $\{b, x\}$ are transversals). This implies $R_b^y C_{a}^{x} R_\sq^b  C_{\sq}^{a} E \neq C_{a}^{x} R_b^y R_\sq^b  C_{\sq}^{a} E$, so we are in case (ii). If instead $y|a \in D$, proceed to (4). The final possibility, that $y|a \notin D$ and $y|x \notin D$, cannot happen by Definition~\ref{defn:corner-system}(c).

\item We can now assume that $R_b^y C_a^x C_{\sq}^a R_\sq^b E = C_a^x R_b^y C_{\sq}^a R_\sq^b E$ and $y|a \in D$ while $b|a \notin D$. Certainly $b|\sq \in R_\sq^b E$, and since $\{y,b\}$ is a transversal we have $y|\sq \notin R_\sq^b E$. We conclude that $C_a^x R_b^y C_{\sq}^a R_\sq^b E \neq C_a^x C_{\sq}^a R_b^y  R_\sq^b E$, so we are in case (ii).
\end{enumerate}
\end{proof}

We end this section with a lemma giving bounds on the partitions appearing in the irreducible decomposition of $V[D]$. Let $\lambda_{\min}$ (resp. $\lambda_{\max}$) be the partition whose row (resp. column) lengths are the sorted row (resp. column) lengths of $D$.
\begin{lem} \label{lem:shape-bounds} If $0 \neq V[\mu] \hookrightarrow V[D]$, then $\lambda_{\min} \leq \mu \leq \lambda_{\max}$ in dominance order. If $V[\lambda_{\max}]$ is nonzero, it is an irreducible factor of $V[D]$ of multiplicity one, and likewise for $V[\lambda_{\min}]$. \end{lem}
\begin{proof} Let $S[D]$ denote the Specht module $\CC[S_D]y_D$. According to \cite[Lemma 3.11]{billey-pawlowski-2013}, if $S[\mu] \hookrightarrow S[D]$ then $\lambda_{\min} \leq \mu \leq \lambda_{\max}$, and $S[\lambda_{\min}]$ and $S[\lambda_{\max}]$ appear in $S[D]$ with multiplicity one. Since $V[D] \simeq V^{\otimes D} \otimes_{\CC[S_D]} S[D]$, one concludes the same for Schur modules modulo the issue that $V[\mu]$ may be zero.
\end{proof}

\begin{cor} \label{cor:nonzero-schur} If $D$ is an ordinary diagram, $V[D]$ is nonzero if and only if every column of $D$ has at most $k = \dim V$ cells. \end{cor}
\begin{proof} Suppose $D$ has a column with more than $k$ cells. Then any filling $T : D \to [k]$ must have $T(\square) = T(\square')$ for distinct cells $\square$ and $\square'$ in the same column of $D$. Set $t = (\square \,\, \square') \in C(D)$. Since $e_T t = e_T$ while $ty_D = -y_D$, we see $e_T y_D = 0$, so $V[D] = V^{\otimes D}y_D = 0$. 

Conversely, if all columns of $D$ have at most $k$ cells, then $\ell(\lambda_{\max}) \leq k$. The Schur module $V[\mu]$ is nonzero if and only if $\ell(\mu) \leq \dim V$, for instance because $s_{\mu}(x_1, \ldots, x_k)$ is nonzero if and only if $\ell(\mu) \leq k$. By Lemma~\ref{lem:shape-bounds}, $V[\lambda_{\max}] \hookrightarrow V[D]$.
\end{proof}

\section{Truncated Schur modules} \label{sec:truncated-schur}
We will show that $G_f = \trunc_{k,n-k} \ch V[D(f)]$, but this phrasing of Theorem~\ref{thm:main} puts an extra conceptual step between $G_f$ and the representation $V[D(f)]$. In this section we modify the notion of Schur module in a way that does not require any knowledge of irreducible decompositions, and provides a natural reason to view the associated characters as members of $\Lambda^{n-k}(k)$.

Fix a complex vector space $V$. Say a degree $d$ tensor $x \in T(V)$ is \emph{$\ell$-symmetric} if the stabilizer of $x$ under the right action of $S_d$ contains the subgroup of $S_d$ stabilizing some $\ell$-subset of $[d]$. Equivalently, viewing $x$ as a $d$-linear form on $V^*$, it is symmetric in at least $\ell$ of its arguments. Given an integer $m \geq 0$, define $I(m)$ to be the span of all $(m+1)$-symmetric tensors. Observe that $I(m)$ is a $\GL(V)$-stable two-sided ideal.

We now proceed as in Section~\ref{subsec:rep}. Define $R^m(V)$ as the Grothendieck group of finite-dimensional $\GL(V)$-submodules of $T(V)/I(m)$. Let $\mu : T(V)/I(m) \otimes T(V)/I(m) \to T(V)/I(m)$ be multiplication in $T(V)/I(m)$. Define the \emph{truncated tensor product} of submodules $U$ and $W$ of $T(V)/I(m)$ to be $U \trncotimes W = \mu(U,W)$, and extend this definition by bilinearity to the case where $U$ and $W$ are direct sums of submodules of $T(V)/I(m)$. The operation $\trncotimes$ makes $R^m(V)$ into a ring. Define the \emph{truncated character} of $[W] \in R^m(V)$ as $\trch(W) := \trunc_{k,m} \ch(W) \in \Lambda^m(k)$.

\begin{lem} \label{lem:l-symmetric-ideal} The ideal $I(m)$ is the sum of all simple $\GL(V)$-submodules of $T(V)$ isomorphic to $V[\lambda]$ where $\lambda_1 > m$. \end{lem}

\begin{proof}
Suppose $U \simeq V[\lambda]$ is a simple submodule of $T(V)$ and $\lambda_1 > m$. Let $F$ be the filling of $\lambda$ labelling every cell in row $i$ with $i$, and identify $e_F \in V[\lambda]$ with a member of $U$. Since $e_F$ is symmetric under permutations of the cells of $\lambda$ in the first row, it is in $U \cap I(m)$. Because $I(m)$ is $\GL(V)$-stable and $U$ is simple, $U = \GL(V)e_F \subseteq I(m)$.

Conversely, we must see that every simple submodule $V[\lambda] \hookrightarrow I(m)$ has $\lambda_1 > m$. For a fixed $(m+1)$-set $K \subseteq \NN$, the space $I^d_K(m)$ of degree $d \geq m+1$ tensors symmetric in positions $I$ is isomorphic as a $\GL(V)$-module to $\Sym^{m+1}(V) \otimes V^{\otimes d-m-1}$. It is easy to see that $\Sym^{m+1}(V)$ is the Schur module $V[m+1]$, so the character of $\Sym^{m+1}(V) \otimes V^{\otimes d-m-1}$ is $s_{m+1} s_1^{d-m-1}$. The Pieri rule shows that every Schur term $s_{\lambda}$ in $s_{m+1} s_1^{d-m-1}$ has $\lambda_1 > m$. Since $I(m)$ is the sum of all the $I^d_K(m)$, the same holds whenever $V[\lambda] \hookrightarrow I(m)$.
\end{proof}

\begin{prop} \label{prop:truncated-character} Let $\pi$ be the map $R(V) \to R^m(V)$ induced by the quotient map $T(V) \to T(V)/I(m)$. Then $\pi$ is surjective with kernel $\vspan \{[V[\lambda]] : \lambda \not\subseteq (m^k)\}$, the truncated character map is a ring isomorphism, and the following diagram commutes:
\begin{equation*}
\begin{CD}
R(V)       @>\ch>> \Lambda(k)\\
@VV\pi V            @VV\trunc_{k,m}V\\
R^m(V)   @>\trch>> \Lambda^m(k)
\end{CD}
\end{equation*}
\end{prop}

\begin{proof}
By definition, $\pi$ is surjective and the diagram commutes. If $W$ is a submodule of $T(V)$, then $\pi([W]) = 0$ if and only if $W \subseteq I(m)$, so the claim about $\ker \pi$ follows from Lemma~\ref{lem:l-symmetric-ideal}. Given $\pi([U]), \pi([W]) \in R^m(V)$, where $[U], [W] \in R(V)$, we have $\pi([U])\trncotimes \pi([W]) = \pi([U]\otimes [W])$ because $T(V) \to T(V)/I(m)$ is a ring homomorphism. But then
\begin{align*}
\trch (\pi([U]) \trncotimes \pi([W])) &= \trch \pi([U \otimes W]) = \trunc_{k,m} \ch([U \otimes W])\\
&= \trunc_{k,m}(\ch U) \trunc_{k,m}(\ch W) = \trch(\pi [U])\trch(\pi [W]),
\end{align*}
so $\trch$ is a ring homomorphism. Finally, $\trch$ is an isomorphism because $\ch$ maps $\ker \pi$ isomorphically onto $\ker \trunc_{k,m}$.
\end{proof}

Now let $D$ be a diagram. The set of $(m+1)$-symmetric tensors in $V^{\otimes D}$ is preserved by the right action of $S_D$, so letting $I_D(m)$ denote their span, there is still a right action of $S_D$ on $V^{\otimes D}/I_D(m)$.

\begin{defn} \label{defn:truncated-schur}
The \emph{truncated Schur module} $V_m[D]$ is $(V^{\otimes D}/I_D(m))y_D$, where $y_D \in S_D$ is the Young symmetrizer of $D$.
\end{defn}
Clearly $[V_m[D]] = \pi [V[D]]$ where $\pi : R(V) \to R^m(V)$ is the map from Proposition~\ref{prop:truncated-character}.

The next lemma is a partial analogue of Corollary~\ref{cor:nonzero-schur} for truncated Schur modules.
\begin{lem} \label{lem:nonzero-truncated-schur} \hfill
\begin{enumerate}[(a)]
\item If a column of $D$ has more than $k = \dim V$ cells, then $V_m[D] = 0$.
\item If a row of $D$ has more than $m$ cells, then $V_m[D] = 0$.
\item If $D$ is contained in a $k \times m$ rectangle, then $V_m[D] \simeq V[D]$.
\end{enumerate}
\end{lem}

\begin{rem}
It is not true that $V_m[D] \neq 0$ if and only if every column of $D$ has at most $k$ cells and every row has at most $m$ cells. For instance, if $D = \{(1,1),(2,2)\}$ and $k = 1$, then $V_1[D] = 0$.
\end{rem}

\begin{proof} \hfill
\begin{enumerate}[(a)]
\item This follows from Corollary~\ref{cor:nonzero-schur}.
\item Suppose $D$ has a row with more than $m$ cells. Then $(\lambda_{\min})_1 > m$, so if $\lambda_{\min} \leq \mu$ then $\mu_1 > m$ also. Lemma~\ref{lem:shape-bounds} and Proposition~\ref{prop:truncated-character} then show $V_m[D] = 0$.
\item If $D$ is contained in a $k \times m$ rectangle, so are $\lambda_{\min}$ and $\lambda_{\max}$, and therefore so is any $\mu$ such that $0 \neq V[\mu] \hookrightarrow V[D]$, by Lemma~\ref{lem:shape-bounds}. This implies $V[D] \cap I_D(m) = 0$, so $V[D] \simeq V_m[D]$.
\end{enumerate}
\end{proof}

\section{Schur modules of Rothe diagrams} \label{sec:main}
With the technical tools of Section~\ref{sec:schur} in hand we can now prove Theorem~\ref{thm:main}, which we will view from the perspective of Section~\ref{sec:truncated-schur}. Recall that $D(f)$ denotes the Rothe diagram of an affine permutation $f$.

\begin{thm} \label{thm:main2}
For any $f \in \Bound(k,n)$, it holds that $G_f = \trch V_{n-k}[D(f)]$ where $\dim V = k$.
\end{thm}

To apply the results of Section~\ref{sec:schur} we need the next lemma.
\begin{lem} \label{lem:toric-rothe-diagram} If $f \in T\Bound(n)$, then $D(f)$ is toric. \end{lem}
\begin{proof}
Suppose $D(f)$ is not toric, so that there exist $(i,f(j)), (i,f(j)+pn) \in D(f)$ where $p > 0$. This means $f(i) > f(j)+pn > f(j)$ and $i < j < j+pn$, so $f(i)-i > f(j)+pn-j$, implying $f \notin T\Bound(n)$.
\end{proof}

Fix $r \in [n]$ and $f \in T\Bound(n)$, and let $f^i = (i, f(i)) \in \Cyl_{n,n}$. Define for each $i \in \ZZ$ a set
\begin{equation*}
\Delta_i = \{f^j \in \Cyl_{n,n} : (i,j) \in \BCov_r(f)\}.
\end{equation*}
Order $\Delta_i$ by column index, so $f^j < f^{j'}$ if $f(j) < f(j')$.
\begin{lem} \label{lem:rothe-diagram-corner} For $i \in [n]$, the pair $(f^i, \Delta_i)$ is a corner configuration for $D(f)$. \end{lem}
\begin{proof} We take the three parts of Definition~\ref{defn:corner} one at a time.
\begin{enumerate}[(a)]
\item Take distinct cells $f^j, f^{j'} \in \Delta_i$. Lemma~\ref{lem:bounded-cover} implies $0 < j-i, j'-i < n$, which means $i, j, j'$ are all distinct modulo $n$. Thus, none of $f^j$, $f^{j'}$, or $f^i$ can share a row or column index, meaning that $\{f^i\} \cup \Delta_i$ is a transversal. If $D(f)$ contained a cell $(q, f(q)+pn)$ then $(q, q+pn)$ would be an inversion of $f$, which is impossible, so $D(f)$ does not meet $\{f^i\} \cup \Delta_i$.

\item Suppose $(i, f(j)+pn) \in D(f)$ for some $p$, so $i < j+pn$ and $f(i) > f(j)+pn$. Since $f(i) < f(j)$, it must be the case that $p < 0$. But then $i < j+pn \leq j-n$, which contradicts Lemma~\ref{lem:bounded-cover}. If $(j, f(i)+pn) \in D(f)$, then $j < i+pn$ implies $p > 0$, so $f(j) > f(i)+pn > f(i)$, which cannot happen since $f < ft_{ij}$. Thus $D(f)$ does not meet $\{f^i\} | \Delta_i$ or $\Delta_i | \{f^i\}$.

\item Say $f(j) < f(j')$. The fact that $\ell(ft_{ij}) = \ell(ft_{ij'}) = \ell(f)+1$ then forces $j' < j$, so $f^{j'}|f^j = (j', f(j))$ is in $D(f)$ as desired.
\end{enumerate}
\end{proof}

Now define $\Cor = \{(f^i, \Delta_i) : \text{$i \in [n]$ and $\Delta_i \neq \emptyset$}\}$. Let $\prec$ be the total order on $[n]$ such that
\begin{itemize}
\item $r+1, \ldots, n \prec 1, \ldots, r$, and;
\item if $r$ is not between $i$ and $i'$, then $i \prec i'$ if and only if $f(i) < f(i')$.
\end{itemize}
Order $\Cor$ so that $(f^i, \Delta_i) < (f^{i'}, \Delta_{i'})$ if and only if $i \prec i'$.
\begin{lem} \label{lem:rothe-diagram-corner-system} The set $\Cor$ forms a system of corner configurations for $D(f)$. \end{lem}
\begin{proof}
Again we deal with parts (a) and (b) of Definition~\ref{defn:corner-system} separately. Take $(f^i, \Delta_i)$ and $(f^{i'}, \Delta_{i'})$ in $\Cor$.
\begin{enumerate}[(a)]
\item Lemma~\ref{lem:rothe-diagram-corner} shows that $(f^i, \Delta_i)$ is a corner configuration, so we only need to show that $f^i$ is not in the same row or column as any $f^{j'} \in \Delta_{i'}$. By definition there exists $f^j \in \Delta_i$. Suppose for the sake of contradiction that $i \equiv j' \pmod{n}$. Since $i, i' \in [n]$, Lemma~\ref{lem:bounded-cover} implies that $j' = i$ or $j' = i+n$.
\begin{itemize}
\item Suppose $j' = i$. We have $j \leq f(j) < f(i) = f(j') < f(i') \leq i' + n$, so $j - i' \leq n$. But then both intervals $[i', j') = [i', i)$ and $[i, j)$ cannot contain $r$ modulo $n$, contradicting $(i,j), (i',j') \in \BCov_r(f)$.
\item Suppose $j' = i+n$. Thus $j \leq f(j) < f(i) = f(j')-n < f(i')-n \leq i'$. But $j \leq i'$ means that both intervals $[i, j)$ and $[i', j') = [i', i+n)$ cannot contain $r$ modulo $n$, again a contradiction.
\end{itemize}

\item Assume $i \prec i'$ and $f^j \in \Delta_i$, $f^{j'} \in \Delta_{i'}$. We want to see that $\{f^{i'}, f^{j'}\} | \{f^i, f^j\}$ meets $D(f)$. Let $I$ be the interval between $i$ and $i'$, and consider four cases.
\begin{itemize}
\item $r \notin I$ and $i' < i$: By definition of $\prec$ we have $f(i) < f(i')$, so $D(f)$ contains $(i', f(i)) = f^{i'} | f^i$.
\item $r \in I$ and $f(i) < f(i')$: Now $i \prec i'$ implies $i' \leq r < i$, so again $D(f)$ contains $(i', f(i))$.
\item $r \notin I$ and $i < i'$: Since $(i,j) \in \BCov_r(f)$, the interval $[i,j)$ contains $r$ or $r+n$ but the interval $I = [i,i')$ does not, so $i < i' < j$. We get the two inequalities $f(i) < f(i'), f(j)$ from $i \prec i'$ and $f < ft_{ij}$, and $\ell(ft_{ij}) = \ell(f)+1$ forces $f(i) < f(j) < f(i')$. But now we see that $(i', f(j)) = f^{i'} | f^j \in D(f)$.
\item $r \in I$ and $f(i') < f(i)$: Again $i \prec i'$ gives $i' \leq r < i$. Both intervals $[i,j)$ and $[i',j')$ contain $r$ modulo $n$, and since $i \prec i'$ implies $i' \leq r < i$ in this case, that can only happen if $r+n \in [i,j)$ and $r \in [i',j')$. Therefore $i < i'+n \leq r+n < j$. As in the previous case, the fact that $\ell(ft_{ij}) = \ell(f)+1$ and $f(i) < f(j), f(i'+n)$ then forces $f(i) < f(j) < f(i'+n)$. We conclude that $(i'+n, f(j)) \in D(f)$, so $f^{i'} | f^j$ meets $D(f)$.
\end{itemize}
\end{enumerate}
\end{proof}

\begin{lem} \label{lem:rothe-jp} If $f^j \in \Delta_i$, then $(D(f) \cup \{f^i\})_{f^i \to f^j}$ is equivalent to $D(ft_{ij})$. \end{lem}
\begin{proof} The diagrams $R_x^{x'} D$ and $R_{x'}^x D$ are always equivalent, namely via swapping the rows of $x$ and $x'$, and likewise for column James-Peel moves, so $(D(f) \cup \{f^i\})_{f^i \to f^j}$ is equivalent to $(D(f) \cup \{f^i\})_{f^j \to f^i}$. Let us see that the latter is equal to $D(ft_{ij})$. This is easiest to see by looking at a picture of $D(f)$:
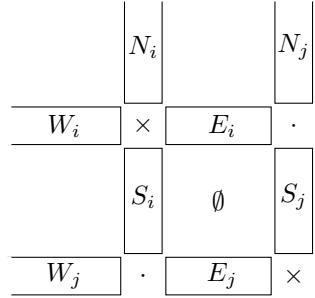
\begin{center}
\begin{tikzpicture}[scale=2]
\node at (0,1) {$\times$};
\node at (1,1) {$\cdot$};
\node at (1,0) {$\times$};
\node at (0,0) {$\cdot$};
\node at (0.5,0.5) {$\emptyset$};

\draw (0.15, 1.125) rectangle (0.85, 0.875);
\node at (0.525, 1) {$E_i$};
\draw (0.15, 0.125) rectangle (0.85, -0.125);
\node at (0.525, 0) {$E_j$};
\draw (-0.15, 1.125) rectangle (-0.9, 0.875);
\node at (-0.525, 1) {$W_i$};
\draw (-0.15, 0.125) rectangle (-0.9, -0.125);
\node at (-0.525, 0) {$W_j$};

\draw (0.125, 0.15) rectangle (-0.125, 0.85);
\node at (0, 0.525) {$S_i$};
\draw (1.125, 0.15) rectangle (0.875, 0.85);
\node at (1, 0.525) {$S_j$};
\draw (0.125, 1.15) rectangle (-0.125, 1.85);
\node at (0, 1.525) {$N_i$};
\draw (1.125, 1.15) rectangle (0.875, 1.85);
\node at (1, 1.525) {$N_j$};

\filldraw[color=white, fill=white] (-0.88, 0.873) rectangle (-0.92, 1.127);
\filldraw[color=white, fill=white] (-0.88, -0.127) rectangle (-0.92, 0.127);
\filldraw[color=white, fill=white] (-0.127, 1.87) rectangle (0.127, 1.83);
\filldraw[color=white, fill=white] (0.873, 1.87) rectangle (1.127, 1.83);
\end{tikzpicture}
\end{center}
Here the two $\times$'s are, from left to right, the points $f^i$ and $f^j$ (which are not in $D(f)$). The region marked $\emptyset$ contains no point $f^p$ because $f < ft_{ij}$ is a Bruhat cover.

Observe that $D(ft_{ij})$ is obtained from $D(f)$ by switching the regions $E_i$ and $E_j$, switching the regions $S_i$ and $S_j$, and adding the cell $f^i$. Because $f \lessdot ft_{ij}$, the region $W_i$ contains $W_j$ in the sense that if $(j, f(p)) \in D(f)$ then $(i, f(p)) \in D(f)$. The region $E_i$ is empty, so the effect of $R_{f^j}^{f^i}$ on $D(f)$ is simply to switch the regions $E_i$ and $E_j$. Analogous arguments show that applying $C_{f^j}^{f^i}$ switches $S_i$ and $S_j$. This shows that $\{f^i\} \cup (D(f)_{f^j \to f^i}) = D(ft_{ij})$, but $\{f^i\} \cup (D(f)_{f^j \to f^i}) = (D(f) \cup \{f^i\})_{f^j \to f^i}$ since $D(f)$ contains none of $f^j$, $f^i$, $f^j|f^i$, or $f^i|f^j$, and $f^j$ shares no row or column index with $f^i$.
\end{proof}

\begin{cor} \label{cor:induction-injection}  If $f \in \Bound(k,n)$ and $r \in \ZZ$, then
\begin{equation*}
\bigoplus_{(i,j) \in \BCov_r(f)} V_{n-k}[D(ft_{ij})] \hookrightarrow V \trncotimes V_{n-k}[D(f)].
\end{equation*}
\end{cor}

\begin{proof}
By Lemma~\ref{lem:rothe-jp}, this is the conclusion of applying Lemma~\ref{lem:corner-system} to the system of corner configurations for $D(f)$ given by Lemma~\ref{lem:rothe-diagram-corner-system}, and then passing to truncated Schur modules. \end{proof}

We will prove by two inductions on $<_r$ that the injection of Corollary~\ref{cor:induction-injection} is an isomorphism, with the next two lemmas serving as base cases. Let $H_f = \trch V_{n-k}[D(f)]$.

\begin{lem} \label{lem:r-bruhat-maximal} $f \in \Bound(k,n)$ is maximal in the order $<_r$ if and only if $\ell(f) = k(n-k)$, in which case $H_f = \bar s_{(n-k)^k} = G_f$. \end{lem}

\begin{proof}
Since the orders $<_r$ are isomorphic for varying $r$, it suffices to consider $r = 0$. Lemma~\ref{lem:0-bruhat-bijection} implies that if $f = f_{u,v}$ is $<_0$-maximal, then $u = v$. By Theorem~\ref{thm:k-bruhat-bijection}, $f_{v,v}$ is also Bruhat-maximal in $\Bound(k,n)$ with length $k(n-k)$. Let $A = \{v(1), \ldots, v(k)\}$ and $B = \{v(k+1), \ldots, v(n)\}$. Then
\begin{equation*}
f_{v,v}(i) = \begin{cases}
i+n & \text{if $i \in A$}\\
i & \text{if $i \in B$}
\end{cases},
\end{equation*}
and so for any $(i,j) \in A \times B$, either $i < j$ and $(i,j) \in D(f)$, or $i-n < j < i$ and $(i-n, j) \in D(f)$. It follows that the image of $D(f)$ under the quotient map $\Cyl_{n,n} \to \ZZ^2 / (n\ZZ \times n\ZZ)$ is simply the $k \times (n-k)$ rectangle $A \times B$, so $V_{n-k}[D(f)] \simeq V_{n-k}[(n-k)^k]$.

As for $G_f$, necessarily $G_f = d s_{(n-k)^k}$ for some $d$, and we must see that $d = 1$. For this we apply Lemma~\ref{lem:affine-stanley-bounds}: the previous paragraph implies $\lambda_{\max} = (n-k)^k$, so $m_{(n-k)^k}$ appears in $\affF_f$ with coefficient $1$. Since this is also the coefficient of $m_{(n-k)^k}$ in $s_{(n-k)^k}$, we get $G_f = \bar s_{(n-k)^k} = H_f$ as desired.
\end{proof}

\begin{lem} \label{lem:r-bruhat-minimal} If $f \in \Bound(k,n)$ is minimal in the order $<_r$, then $H_f = G_f$. \end{lem}
\begin{proof}
Again we can assume $r = 0$. Lemma~\ref{lem:0-bruhat-bijection} shows that the $0$-minimal elements of $\Bound(k,n)$ are those of the form $f_{1, v}$. But note that $g_{k,n}$ and hence $f_{1, v}$ maps $[n]$ onto $[k+1,n+k]$, so $w := \tau^{-k} f_{1,v} \in S_n$. By \cite[Proposition 5]{lam-affine-stanley}, $\affF_w$ agrees with the ordinary Stanley symmetric function $F_w$, while $F_w = \ch V[D(f)]$ by \cite[Theorem 33]{plactification}. Since $\affF_w = \affF_{{\tau}^k w}$ we are done.
\end{proof}

Given a partition $\lambda \subseteq (n-k)^k$, let $\lambda^{\vee}$ be the partition with Young diagram $([k] \times [n-k]) \setminus \lambda$ (rotated $180^\circ$). Let $\delta : \Lambda^{n-k}(k) \to \ZZ$ be the linear map with $\delta(\bar s_{\lambda}) = f^{\lambda^{\vee}}$,  where $f^{\lambda^{\vee}}$ is the number of standard Young tableaux of shape $\lambda^{\vee}$. 

\begin{lem} \label{lem:delta} If $F \in \Lambda^{n-k}(k)$ has degree less than $k(n-k)$, then $\delta(\bar s_1 F) = \delta(F)$. \end{lem}
\begin{proof}  It suffices to assume that $F = \bar s_{\lambda}$ where $|\lambda| < k(n-k)$. By Pieri's rule, $\delta(\bar s_1 \bar s_{\lambda}) = f^{\lambda^\vee \setminus 1}$. But $f^{\lambda^\vee \setminus 1} = f^{\lambda^\vee}$: removing the box labeled $1$ gives a bijection from the tableaux counted by $f^{\lambda^\vee}$ to those counted by $f^{\lambda^\vee \setminus 1}$.
\end{proof}

\begin{rem} Under the isomorphism $\Lambda^{n-k}(k) \simeq H^*(\Gr(k,n), \ZZ)$, $\delta$ sends a cohomology class $[X]$ to the degree of $X$ as a projective subvariety, a point of view which makes Lemma~\ref{lem:delta} clear. \end{rem}

\begin{lem} \label{lem:induction-isomorphism} For any $f \in \Bound(k,n)$, $\delta(H_f) = \delta(G_f)$ and
\begin{equation*}
\bigoplus_{(i,j) \in \BCov_r(f)} V_{n-k}[D(ft_{ij})] \simeq V \trncotimes V_{n-k}[D(f)].
\end{equation*}
\end{lem}
\begin{proof}
First let us see that $\delta(H_f) \geq \delta(G_f)$. Induct downwards on $\ell(f)$. If $\ell(f)$ is as large as possible, i.e. $\ell(f) = k(n-k)$, then $\delta(H_f) = 1 = \delta(G_f)$ by Lemma~\ref{lem:r-bruhat-maximal}. If $\ell(f) < k(n-k)$, then
\begin{equation} \label{eq:delta-inequalities}
\delta(H_f) \geq \sum_{(i,j) \in \BCov_r(f)} \delta(H_{ft_{ij}}) \geq \sum_{(i,j) \in \BCov_r(f)} \delta(G_{ft_{ij}}) = \delta(G_f),
\end{equation}
where the three relations above hold by Corollary~\ref{cor:induction-injection} plus Lemma~\ref{lem:delta}, induction, and Proposition~\ref{prop:recurrences} respectively.

Suppose for the moment that $\delta(H_f) = \delta(G_f)$ for some particular $f$. Then all inequalities in \eqref{eq:delta-inequalities} are equalities. In particular, $\sum_{(i,j) \in \BCov_r(f)} \delta(H_{ft_{ij}}) = \sum_{(i,j) \in \BCov_r(f)} \delta(G_{ft_{ij}})$, and since $\delta(H_{ft_{ij}}) \geq \delta(G_{ft_{ij}})$ by the previous paragraph, $\delta(H_{ft_{ij}}) = \delta(G_{ft_{ij}})$ for all such $(i,j)$. Thus, if $\delta(H_f) = \delta(G_f)$ for any particular $f$, then $\delta(H_g) = \delta(G_g)$ for any $g >_r f$. Now apply Lemma~\ref{lem:r-bruhat-minimal}. We conclude that $\delta(H_f) = \delta(G_f)$ and that all inequalities in \eqref{eq:delta-inequalities} are equalities, which implies the isomorphism claimed in the lemma.
\end{proof}

\begin{thm*}[Restatement of Theorem~\ref{thm:main2}] For any $f \in \Bound(k,n)$, $G_f = H_f = \trch V_{n-k}[D(f)]$. \end{thm*}
\begin{proof} Lemma~\ref{lem:induction-isomorphism} shows that $H_f$ satisfies the bounded affine L-S recurrence of Proposition~\ref{prop:LS-recurrence} which $G_f$ satisfies, so by Theorem~\ref{thm:finite-LS-tree} and Lemma~\ref{lem:schur-leaves}, it suffices to show that $H_f = \bar s_{\lambda} = G_f$ when $f$ is $0$-Grassmannian of shape $\lambda$. But in that case $D(f)$ is equivalent to $\lambda$, so $H_f = \trch V_{n-k}[\lambda] = \bar s_{\lambda}$.
\end{proof}

\section{Applications} \label{sec:applications}

\subsection{Toric Schur polynomials} \label{subsec:toric-schur}

A \emph{closed lattice path} $P$ in $\Cyl_{k,n-k}$ is a circular sequence $(p_1, \ldots, p_{n})$ labeled by $\ZZ/n\ZZ$ such that $p_{i+1} - p_i \in \{(\pm 1,0), (0,\pm 1)\}$ for all $i$. If $p_{i+1} - p_i \in \{(1,0), (0,1)\}$ for all $i$, we say $P$ moves from northwest to southeast. We think of $P$ as the path obtained by concatenating the line segments from $p_i$ to $p_{i+1}$ for all $i$.
\begin{defn}
A \emph{cylindric skew shape} is the set of unit squares $[i,i+1] \times [j,j+1]$ in a cylinder $\Cyl_{k,n-k}$ between two closed lattice paths moving from northwest to southeast which do not cross (though they can meet).
\end{defn}
Any cylindric skew shape is a cylindric diagram. A filling of a cylindric skew shape $\Theta$ by positive integers is a \emph{semistandard cylindric tableau} if it is weakly increasing rightward across rows, and strictly increasing up columns (recall that we are drawing partitions in the French style). As usual, a semistandard tableau is \emph{standard} if it uses exactly the integers $1, 2, \ldots, |\Theta|$.

\begin{ex} \label{ex:cylindric-shape}
Here is a standard tableau for a cylindric skew shape with $12$ boxes in $\Cyl_{4,5}$ (more precisely, we draw part of the inverse image of the tableau in $\ZZ^2$ under the quotient map to the cylinder):
\begin{center}
\begin{tikzpicture}[scale=.5]
\begin{scope}
    \clip (-3,4) -- (-3,3) -- (-2,3) -- (-2,2) -- (-1,2) -- (-1,1) -- (0,1) -- (1,1) -- (1,0) -- (2,0) -- (2,-1) -- (3,-1) -- (3,-2) -- (4,-2) -- (4,-3) -- (5,-3)       --      (5,-4)     --    (3, -4) -- (3, -3) -- (2, -3) -- (2, -2) -- (1, -2) -- (0, -2) -- (0, -1) -- (-1, -1) -- (-2, -1) -- (-2, 0) -- (-2, 1) -- (-3, 1) -- (-3, 2) -- (-4, 2) -- (-5, 2) --       (-5,4) -- cycle;
    \filldraw[fill=lightgray!70!white] (-5,-4) -- (-5,4) -- (5,4) -- (5,-4) -- cycle;
\end{scope}
\draw (-5,-4) grid (5,4);
\draw[color=blue!80!black, very thick] (-5,2) -- (-4,2) -- (-3,2) -- (-3,1) -- (-2,1) -- (-2,0) -- (-2,-1) -- (-1,-1) -- (0,-1) -- (0,-2) -- (1,-2) -- (2,-2) -- (2,-3) -- (3,-3) -- (3,-4);
\draw[color=red!80!white, very thick] (-3,4) -- (-3,3) -- (-2,3) -- (-2,2) -- (-1,2) -- (-1,1) -- (0,1) -- (1,1) -- (1,0) -- (2,0) -- (2,-1) -- (3,-1) -- (3,-2) -- (4,-2) -- (4,-3) -- (5,-3);

\node at (-4.5,3.5) {$7$};
\node at (-3.5,3.5) {$11$};

\node at (-4.5,2.5) {$6$};
\node at (-3.5,2.5) {$10$};
\node at (-2.5,2.5) {$12$};

\node at (-2.5,1.5) {$4$};
\node at (-1.5,1.5) {$9$};

\node at (-1.5,0.5) {$2$};
\node at (-0.5,0.5) {$5$};
\node at (0.5,0.5) {$8$};

\node at (-1.5,-0.5) {$1$};
\node at (-0.5,-0.5) {$3$};
\node at (0.5,-0.5) {$7$};
\node at (1.5,-0.5) {$11$};

\node at (0.5,-1.5) {$6$};
\node at (1.5,-1.5) {$10$};
\node at (2.5,-1.5) {$12$};

\node at (2.5,-2.5) {$4$};
\node at (3.5,-2.5) {$9$};

\node at (3.5,-3.5) {$2$};
\node at (4.5,-3.5) {$5$};
\end{tikzpicture}
\end{center}
\end{ex}

\begin{defn}
The \emph{cylindric Schur function} associated to a cylindric skew shape $\Theta$ is the formal power series $s_\Theta := \sum_T x^T$, where $T$ runs over semistandard tableaux of shape $\Theta$, and as usual $x^T$ means $\prod_{c \in \Theta} x_{T(c)}$.
\end{defn}

As in Section~\ref{sec:schur}, call a cylindric skew shape $\Theta$ \emph{toric} if the quotient map $\Cyl_{k,n-k} \to \ZZ^2 / (k\ZZ \times (n-k)\ZZ)$ is injective on $\Theta$. For a toric shape $\Theta \subseteq \Cyl_{k,n-k}$, call the polynomial $s_\Theta(X_k)$ a \emph{toric Schur polynomial}. Cylindric Schur functions were introduced by Postnikov \cite{postnikov-affine-quantum-schubert}, and the next theorem summarizes some of his main results.

\begin{thm}[\cite{postnikov-affine-quantum-schubert}]  \label{thm:toric-schur-facts}
Let $\Theta \subseteq \Cyl_{k,n-k}$ be a cylindric skew shape.
\begin{enumerate}[(a)]
\item $s_\Theta$ is a symmetric function;
\item $s_\Theta(X_k)$ is Schur-positive, and is nonzero if and only if $\Theta$ is toric;
\item For $\Theta$ ranging over toric shapes in $\Cyl_{k,n-k}$, the Schur coefficients of $s_\Theta(X_k)$ are exactly the 3-point Gromov-Witten invariants for $\Gr(k,n)$.
\end{enumerate}
\end{thm}
When $\Theta$ is an ordinary skew shape $\lambda \setminus \mu$, $s_\Theta$ is simply the usual skew Schur function $s_{\lambda \setminus \mu}$, and $s_{\lambda \setminus \mu}(X_k) = \ch V[\lambda \setminus \mu]$. More generally, Postnikov conjectured that for any toric skew shape $\Theta$, the toric Schur polynomial $s_\Theta(X_k)$ is the character of the Schur module $V[\Theta]$ \cite[Conjecture 10.1]{postnikov-affine-quantum-schubert}. We now explain how Postnikov's conjecture follows from Theorem~\ref{thm:main2}. By Theorem~\ref{thm:toric-schur-facts}(c), this yields a representation-theoretic interpretation of the 3-point Gromov-Witten invariants for $\Gr(k,n)$.

\begin{thm}[\cite{lam-affine-stanley}] \label{thm:cylindric-affine-stanley} For any cylindric skew shape $\Theta$, there is an affine permutation $f_\Theta$ with $s_\Theta = \affF_{f_\Theta}$. \end{thm}
We will need an explicit description of $f_{\Theta}$, which we take from \cite[Section 8]{positroidjuggling}. Our cylindric shapes are reflections through a horizontal line of those in \cite{postnikov-affine-quantum-schubert, positroidjuggling}---this is so that the Rothe diagram of $f_\Theta$ is equivalent to $\Theta$. Label the lower and upper boundaries of $\Theta$ with $\ZZ$, increasing from northwest to southeast. If $i \in \ZZ$ is the label of a vertical edge of the lower boundary, let $f_\Theta(i)$ be the label of the vertical edge of the upper boundary in the same row; likewise if $i$ labels a horizontal edge of the lower boundary, let $f_\Theta(i)$ be the label of the horizontal edge of the upper boundary in the same column. This determines $f_\Theta$ up to left and right multiplication by $\tau$, and so $\affF_{f_\Theta}$ is well-defined, as is $D(f_\Theta)$ up to equivalence.

\begin{ex} \label{ex:affine-perm-from-cylindric-shape}
Take $\Theta$ to be the shape from Example~\ref{ex:cylindric-shape}, with the following edge labeling:
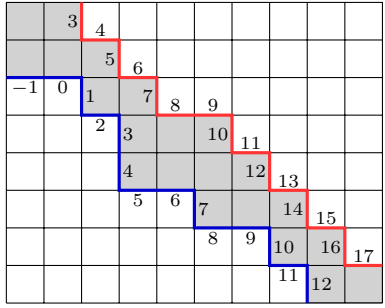
\begin{center}
\begin{tikzpicture}[scale=.5]
\begin{scope}
    \clip (-3,4) -- (-3,3) -- (-2,3) -- (-2,2) -- (-1,2) -- (-1,1) -- (0,1) -- (1,1) -- (1,0) -- (2,0) -- (2,-1) -- (3,-1) -- (3,-2) -- (4,-2) -- (4,-3) -- (5,-3)       --      (5,-4)     --    (3, -4) -- (3, -3) -- (2, -3) -- (2, -2) -- (1, -2) -- (0, -2) -- (0, -1) -- (-1, -1) -- (-2, -1) -- (-2, 0) -- (-2, 1) -- (-3, 1) -- (-3, 2) -- (-4, 2) -- (-5, 2) --       (-5,4) -- cycle;
    \filldraw[fill=lightgray!70!white] (-5,-4) -- (-5,4) -- (5,4) -- (5,-4) -- cycle;
\end{scope}
\draw (-5,-4) grid (5,4);
\draw[color=blue!80!black, very thick] (-5,2) -- (-4,2) -- (-3,2) -- (-3,1) -- (-2,1) -- (-2,0) -- (-2,-1) -- (-1,-1) -- (0,-1) -- (0,-2) -- (1,-2) -- (2,-2) -- (2,-3) -- (3,-3) -- (3,-4);
\draw[color=red!80!white, very thick] (-3,4) -- (-3,3) -- (-2,3) -- (-2,2) -- (-1,2) -- (-1,1) -- (0,1) -- (1,1) -- (1,0) -- (2,0) -- (2,-1) -- (3,-1) -- (3,-2) -- (4,-2) -- (4,-3) -- (5,-3);

\node[below] at (-4.5, 2+0.15) {$\scriptstyle -1$};
\node[below] at (-3.5, 2+0.15) {$\scriptstyle 0$};
\node[right] at (-3-0.15, 1.5) {$\scriptstyle 1$};
\node[below] at (-2.5, 1+0.15) {$\scriptstyle 2$};
\node[right] at (-2-0.15, 0.5) {$\scriptstyle 3$};
\node[right] at (-2-0.15, -0.5) {$\scriptstyle 4$};
\node[below] at (-1.5, -1+0.15) {$\scriptstyle 5$};
\node[below] at (-0.5, -1+0.15) {$\scriptstyle 6$};
\node[right] at (0-0.15, -1.5) {$\scriptstyle 7$};
\node[below] at (0.5, -2+0.15) {$\scriptstyle 8$};
\node[below] at (1.5, -2+0.15) {$\scriptstyle 9$};
\node[right] at (2-0.15, -2.5) {$\scriptstyle 10$};
\node[below] at (2.5, -3+0.15) {$\scriptstyle 11$};
\node[right] at (3-0.15, -3.5) {$\scriptstyle 12$};

\node[left] at (-3+.15, 3.5) {$\scriptstyle 3$};
\node[above] at (-2.5, 3-.15) {$\scriptstyle 4$};
\node[left] at (-2+.15, 2.5) {$\scriptstyle 5$};
\node[above] at (-1.5, 2-.15) {$\scriptstyle 6$};
\node[left] at (-1+.15, 1.5) {$\scriptstyle 7$};
\node[above] at (-0.5, 1-.15) {$\scriptstyle 8$};
\node[above] at (0.5, 1-.15) {$\scriptstyle 9$};
\node[left] at (1+.15, 0.5) {$\scriptstyle 10$};
\node[above] at (1.5, 0-.15) {$\scriptstyle 11$};
\node[left] at (2+.15, -0.5) {$\scriptstyle 12$};
\node[above] at (2.5, -1.0-.15) {$\scriptstyle 13$};
\node[left] at (3+.15, -1.5) {$\scriptstyle 14$};
\node[above] at (3.5, -2.0-.15) {$\scriptstyle 15$};
\node[left] at (4+.15, -2.5) {$\scriptstyle 16$};
\node[above] at (4.5, -3.0-.15) {$\scriptstyle 17$};
\end{tikzpicture}
\end{center}
Then $f_{\Theta} = 74(10)(12)68(14)9(11)$, and we have chosen the edge labeling so that $f_{\Theta} \in \Bound(4,9)$.
\end{ex}

\begin{lem} \label{lem:cylindric-rothe-diagram} For any cylindric diagram $\Theta \subseteq \Cyl_{k,n-k}$, the Rothe diagram $D(f_\Theta)$ is equivalent to $\Theta$. \end{lem}
\begin{proof}
Consider the injective map $\Theta \to \Cyl_{n,n}$ sending a cell $c$ to $(i,j)$, where $i$ is the label of the lower boundary in the row containing $c$, and $j$ the label of the upper boundary in the column containing $c$. One checks that its image is $D(f_\Theta)$, and evidently it preserves the partitions of cells into rows and columns.
\end{proof}

\begin{thm} \label{thm:toric-schur-module} Let $\Theta \subseteq \Cyl_{k,n-k}$ be a toric diagram. Then $s_\Theta(X_k) = \ch V[\Theta]$. \end{thm}
\begin{proof}
Theorem~\ref{thm:cylindric-affine-stanley}, Theorem~\ref{thm:main2}, and Lemma~\ref{lem:cylindric-rothe-diagram} respectively give the three equalities in
\begin{equation*}
\trunc_{k,n-k} s_\Theta(X_k) = G_{f_\Theta} = \trch V_{n-k}[D(f_\Theta)] = \trch V_{n-k}[\Theta].
\end{equation*}
Because $\Theta$ is toric, it can be viewed as a subset of $[k] \times [n-k]$. Lemma~\ref{lem:affine-stanley-bounds} then implies that every $\lambda$ such that $s_{\lambda}(X_k)$ is a Schur term of $s_\Theta(X_k)$ satisfies $\lambda \subseteq (n-k)^k$, and Lemma~\ref{lem:shape-bounds} implies the same for $\ch V[\Theta]$. From this and the equality $\trunc_{k,n-k} s_{\Theta}(X_k) = \trch V_{n-k}[\Theta]$, we deduce $s_{\Theta}(X_k) = \ch V[\Theta]$.
\end{proof}

\subsection{Schubert times Schur coefficients}
Let $\Fl(n)$ be the variety of complete flags in $\CC^n$, whose elements are chains $F_{\bullet} = (F_1 \subseteq F_2 \subseteq \cdots \subseteq F_n)$ of linear subspaces of $\CC^n$ with $\dim F_i = i$. Choose a basis $e_1, \ldots, e_n$ of $\CC^n$, and for $w \in S_n$ let $w_{\bullet}$ be the flag with $w_i = \langle e_{w(1)}, \ldots, e_{w(i)} \rangle$. To each $w \in S_n$ one associates its \emph{Schubert variety}
\begin{equation*}
X_w = \{F_\bullet : \dim(F_i \cap (w_0)_j) \geq \dim(w_i \cap (w_0)_j) \text{ for all $i, j \in [n]$}\}
\end{equation*}
and \emph{opposite Schubert variety}
\begin{equation*}
X^w = \{F_\bullet : \dim(F_i \cap \id_j) \geq \dim(w_i \cap \id_j) \text{ for all $i, j \in [n]$}\} = X_{ww_0} w_0.
\end{equation*}
Here $\id = 12\cdots n$ and $w_0 = n\cdots 21$, and $S_n$ acts on $\CC^n$ (and therefore $\Fl(n)$) on the right by permuting $e_1, \ldots, e_n$.

Being a closed subvariety of $\Fl(n)$, each Schubert variety $X_w$ has a Poincar\'e dual cohomology class $[X_w] \in H^*(\Fl(n), \ZZ)$ (alternatively, $H^*(\Fl(n), \ZZ)$ is isomorphic to the Chow ring of $\Fl(n)$ and $[X_w]$ is the Chow class of $X_w$). Borel showed that $H^*(\Fl(n), \ZZ)$ is isomorphic to a quotient of the polynomial ring $\ZZ[x_1, \ldots, x_n]$, and the \emph{Schubert polynomial} $\fS_w \in \ZZ[x_1, \ldots, x_n]$ is a polynomial representing $[X_w]$ under this isomorphism. For details, including an explicit definition of $\fS_w$, see \cite{youngtableaux} or \cite{manivel}.

The Schubert classes $[X_u]$ form a $\ZZ$-basis of $H^*(\Fl(n), \ZZ)$; let $c_{v,w}^u$ be the coefficient of $[X_u]$ in $[X_v][X_w]$. Equivalently, $c_{v,w}^u$ is the coefficient of $\fS_u$ in $\fS_v \fS_w$. The structure constants $c_{v,w}^u$ are nonnegative, and it is a major open problem to describe them combinatorially. We elaborate on the geometric reasons for this nonnegativity, since we will use a similar argument in Proposition~\ref{prop:schubert-coefficients} below. The varieties $X_v$ and $X^w$ always intersect generically transversely in the irreducible \emph{Richardson variety} $X_v \cap X^w$, and $X_w \cap X^w$ is a point. It follows that $[X_v][X_w][X^u]$ is $c_{v,w}^u$ times the class of a point. On the other hand, $[X_w][X^u] = [X_w \cap X^u]$ by the aforementioned transversality, and for generic $g \in \GL(n)$, Kleiman's transversality theorem implies $[X_v g][X_w][X^u] = [X^v g \cap X_w \cap X^u]$. Thus, $c_{v,w}^u$ is the size of the finite set $X_v g \cap X_w \cap X^u$.

The connection between Schubert polynomials and the functions $G_f \in \Lambda^{n-k}(k)$ is now as follows.
\begin{thm}[\cite{positroidjuggling}] \label{thm:richardson-projection}
Suppose $u \leq_k v$ are in $S_n$. Then $\pi$ maps an open dense subset of $X_u \cap X^v$ isomorphically onto an open dense subset of $\pi(X_u \cap X^v)$, and $G_{f_{u,v}}$ represents the cohomology class $[\pi(X_u \cap X^v)]$.
\end{thm}

Just as in the flag variety, one has Schubert varieties $X_{\lambda}$ and opposite Schubert varieties $X^{\lambda}$ in the Grassmannian $\Gr(k,n)$, now indexed by partitions $\lambda \subseteq (n-k)^k$, with the relationship $X^{\lambda} = X_{\lambda^\vee} w_0$. The classes $[X_{\lambda}]$ again form a basis of $H^*(\Gr(k,n), \ZZ)$, and correspond to the Schur functions $\bar s_\lambda$. The varieties $X_{\lambda}$ and $X^{\mu}$ are generically transverse, and $X_{\lambda} \cap X^{\lambda}$ is a point. Letting $\pi : \Fl(n) \to \Gr(k,n)$ be the projection map sending $F_\bullet \mapsto F_k$, one has $\pi^{-1}(X_{\lambda}) = X_{w(\lambda)}$, where $w(\lambda)$ is the unique $k$-Grassmannian permutation in $S_n$ with Rothe diagram equivalent to $\lambda$. This leads directly to the next interpretation of the Schur coefficients of $G_f$. 
\begin{prop} \label{prop:schubert-coefficients}
For any $u \leq_k v$ in $S_n$,
\begin{equation*}
G_{f_{u,v}} = \sum_{\lambda \subseteq (n-k)^k} c_{u, w(\lambda)}^v s_{\lambda^\vee}.
\end{equation*}
\end{prop}

\begin{proof}
Let $a_{u,v}^{\lambda}$ be the coefficient of $s_{\lambda^\vee}$ in $G_{f_{u,v}}$, or equivalently by Theorem~\ref{thm:richardson-projection}, the coefficient of $[X_{\lambda^\vee}]$ in $[\Pi_{f_{u,v}}]$. Write $[*]$ for the class of a point in $\Gr(k,n)$ or $\Fl(n)$ as appropriate. On the one hand,
\begin{align*}
\pi_*(\pi^*([X^{\lambda^\vee}])[X_u \cap X^v]) &= \pi_*([X_{w(\lambda)}][X_u \cap X^v])\\
& = \pi_*(c_{u,w(\lambda)}^v[*]) = c_{u,w(\lambda)}^v[*].
\end{align*}
On the other hand, the push-pull formula gives
\begin{equation*}
\pi_*(\pi^*([X^{\lambda^\vee}])[X_u \cap X^v]) = [X^{\lambda^\vee}] \pi_*([X_u \cap X^v]) = [X^{\lambda^\vee}][\Pi_{f_{u,v}}],
\end{equation*}
where the second equality follows from Theorem~\ref{thm:richardson-projection}. The remarks above the statement of the proposition show this is the same as $a_{u,v}^{\lambda}[*]$.
\end{proof}

The Schubert polynomial $\fS_{w(\lambda)}$ is simply $s_{\lambda}(x_1, \ldots, x_k)$. By \cite[Proposition 1.1.1]{k-bruhat-order}, if $c^v_{u,w(\lambda)} \neq 0$ then $u \leq_k v$, so Proposition~\ref{prop:schubert-coefficients} means that knowing the Schur coefficients of the functions $G_f$ is equivalent to knowing the Schubert coefficients of the products $\fS_u \fS_{w(\lambda)}$ for arbitrary $u$. Thus, Theorem~\ref{thm:main2} gives a new, and manifestly positive, representation-theoretic interpretation of these Schubert coefficients.

\subsection{Branching rules for Specht modules} \label{subsec:branching}
Given a diagram $D$ with $d$ boxes, let $L_D$ be the complex vector space with a basis indexed by the bijective labellings $D \to [d]$. There is a natural right action of $S_D$ on $L_D$, and a natural left action of $S_d$. The \emph{generalized Specht module} of $D$ is the left $S_d$-module $S[D] := L_D y_D$. Observe that $V[D] \simeq V^{\otimes d} \otimes_{\CC[S_d]} S[D]$.

 The \emph{Frobenius characteristic} $\fch(M)$ of an $S_d$-module $M$ is the symmetric function
\begin{equation*}
\lim_{n \to \infty} \ch((\CC^n)^{\otimes d} \otimes_{\CC[S_d]} M).
\end{equation*}
For instance, $\ch S[\lambda] = s_{\lambda}$. It is well-known that the Specht modules $S[\lambda]$ for $\lambda$ a partition of $d$ are exactly the irreducible complex representations of $S_d$ \cite[Ch. 7]{youngtableaux}, so that one can compute the irreducible decomposition of $M$ by computing the Schur expansion of $\ch M$. Let $R_{S_d}$ be the Grothendieck group of representations of $S_d$, and set $R = \bigoplus_{d=0}^{\infty} R_{S_d}$.

Define a linear map $\SS^{\vee} : R^{m}(V) \to R$ by $\SS [V[\lambda]] = [S[\lambda^{\vee}]]$ where $\lambda \subseteq m^k$.

\begin{rem}
One can construct $\SS^{\vee}$ in a more natural way as follows. Let $\det^r$ be the 1-dimensional representation of $\GL(V)$ on which $g$ acts as multiplication by $\det(g)^r$. For a submodule $U \subseteq T(V)/I(m)$, sending $U \mapsto U^{\vee} := \Hom(U, \det^{m})$ defines the linear injection $R^{m}(V) \to R(V)$ with $[V_{m}(\lambda)] \mapsto [V[\lambda^{\vee}]]$.

Now suppose $U \subseteq T(V)$ is a homogeneous submodule of degree $d$. Choose a vector space $V'$ with $\dim V' \geq d$ and an injection $\iota : V \hookrightarrow V'$. Then $\iota$ induces an injection $T(V) \hookrightarrow T(V')$, which we will also call $\iota$. Now let $\SS U$ be the $1^d$-weight space of $\GL(V')\iota(U)$ with the action of $S_d \subseteq \GL(V')$ by permutation matrices. By considering highest-weight vectors one checks that $\GL(V')\iota(V[\lambda]) = V'[\lambda]$, and the $1^d$-weight space in $(V')^{\otimes d}$ is isomorphic to $\CC[S_d]$ as an $S_d$-bimodule. Thus, $\SS(U^{\vee})$ as defined using the constructions outlined here matches $\SS^{\vee}U$ as defined above, up to isomorphism.
\end{rem}

For an $S_d$-module $M$, let $\Res^{S_d}_{S_{d-1}} M$ denote $M$ viewed as an $S_{d-1}$-module, where we identify $S_{d-1}$ as the subgroup of $S_d$ fixing $d$. A \emph{branching rule} for $M$ is a formula $\Res^{S_d}_{S_{d-1}} M \simeq \bigoplus_i M_i$ for some $S_{d-1}$-modules $M_i$. In this subsection we show how Lemma~\ref{lem:induction-isomorphism} leads to a branching rule for the $S_{\ell(f)}$-modules $\SS^\vee V_{n-k}[D(f)]$ for $f \in \Bound(k,n)$. 

\begin{lem} \label{lem:branching}
If $[U] \in R^{m}(V)$ is homogeneous of degree $d$, then $\SS^\vee (U \trncotimes V) \simeq \Res^{S_d}_{S_{d-1}} \SS^\vee U$.
\end{lem}

\begin{proof}
It suffices to take $U = V[\lambda^\vee]$ and show that both sides have the same Frobenius characteristic. By \cite[\S 7.3]{youngtableaux},
\begin{equation*}
\fch \Res^{S_d}_{S_{d-1}} S[\lambda] = s_{\lambda / 1} = \sum_{\mu} s_{\mu},
\end{equation*}
where $\mu$ runs over partitions obtained from $\lambda$ by removing one box. On the other hand, Pieri's rule says
\begin{equation*}
V_m[\lambda^{\vee}] \trncotimes V \simeq \bigoplus_{\nu} V_m[\nu]
\end{equation*}
where $\nu$ runs over partitions obtained from $\lambda^{\vee}$ by adding one box such that $\nu \subseteq (m^k)$. Since these two sets $\{\mu\}$ and $\{\nu\}$ of partitions are related by the box complement operation $^\vee$, the lemma follows.
\end{proof}

Define $M_f = \SS^{\vee} V_{n-k}[D(f)]$, a representation of $S_{\ell(f)}$. According to Lemma~\ref{lem:branching}, Lemma~\ref{lem:induction-isomorphism} immediately translates into the following branching rule for the modules $M_f$.
\begin{thm} \label{thm:branching}
For any $f \in \Bound(k,n)$ with length $\ell = \ell(f)$, and any $r \in \ZZ$, one has
\begin{equation*}
\Res^{S_\ell}_{S_{\ell-1}} M_f \simeq \bigoplus_{(i,j) \in \BCov_r(f)} M_{ft_{ij}}.
\end{equation*}
\end{thm}

\begin{cor}
If $u \leq_k v$, the dimension of $M_{f_{u,v}}$ is the number of maximal chains in the $k$-Bruhat interval $[u,v]_k$.
\end{cor}
\begin{proof}
Let $c$ be the number of maximal chains in $[u,v]_k$. By Lemma~\ref{lem:0-bruhat-bijection}, Theorem~\ref{thm:branching} is equivalent to
\begin{equation*}
\Res^{S_\ell}_{S_{\ell-1}} M_{f_{u,v}} \simeq \bigoplus_{u \lessdot_k u' \leq_k v} M_{f_{u', v}}.
\end{equation*}
Iterating this rule gives $\Res^{S_\ell}_{S_0} M_{f_{v,v}} \simeq c M_{f_{v,v}}$, and Lemma~\ref{lem:r-bruhat-maximal} shows that
\begin{equation*}
M_{f_{v,v}} = \SS^\vee V_{n-k}[(n-k)^k] = S[\emptyset]
\end{equation*}
is $1$-dimensional.
\end{proof}

In fact, the symmetric function $\fch M_{f_{u,v}}$ appears in work of Bergeron and Sottile as a quasisymmetric generating function for maximal chains in $[u,v]_k$. Given an integer $\ell$ and $D \subseteq [\ell-1]$, the associated \emph{fundamental quasisymmetric function} is
\begin{equation*}
Q_D = \sum_{\substack{1 \leq i_1 \leq \cdots \leq i_\ell \\ j \in D \Rightarrow i_j < i_{j+1}} } x_{i_1}\cdots x_{i_{\ell}}.
\end{equation*}
Label each covering relation $u \lessdot_k t_{ij}u$ by the integer $j$, and associate to a saturated chain in $k$-Bruhat order the word obtained by concatenating the labels of the covers involved. The \emph{descent set} $\Des(c)$ of a chain $c$ of length $\ell$ is then the descent set of the corresponding word $a$, namely $\Des(c) = \{i \in [\ell-1] : a_i > a_{i+1}\}$. Now define a power series
\begin{equation*}
K_{u,v} = \sum_{c} Q_{\Des(c)}
\end{equation*}
where $c$ runs over the maximal chains in $[u,v]_k$.

\begin{thm}[\cite{bergeron-sottile-skew-schubert}, \cite{assaf-bergeron-sottile}] \label{thm:bergeron-sottile}
The power series $K_{u,v}$ is symmetric, and the coefficient of $s_{\lambda}$ is the Schubert structure constant $c_{u,w(\lambda)}^{v}$.
\end{thm}

By Proposition~\ref{prop:schubert-coefficients}, Theorem~\ref{thm:bergeron-sottile} immediately implies the next corollary.
\begin{cor}
$K_{u,v} = \fch M_{f_{u,v}}$ for any $u \leq_k v$.
\end{cor}

\subsection{Three-row diagrams}
By a \emph{$k$-row diagram} we mean an ordinary diagram with at most $k$ non-empty rows. Any two-row diagram $D$ is equivalent to a skew shape, and then the irreducible decomposition of $V[D]$ is easily computed from Pieri's rule. In this section we observe that every three-row diagram $D$ is (essentially) a toric skew shape, so that the results of Sections~\ref{sec:recurrences} and \ref{subsec:toric-schur} give a combinatorial algorithm for decomposing $V[D]$ into irreducibles. For a different approach, see \cite{triangles}, which describes an algorithm for writing $\ch V[D]$ as a sum of rational functions whenever $D$ is a three-row diagram.

\begin{lem} \label{lem:full-column}
Suppose $D$ is a diagram contained in $[k] \times [m]$ and that $c \subseteq D$ is a column of size $k$. If $V[D \setminus c] \simeq \bigoplus_{\lambda} c_{\lambda}V[\lambda]$, then $V[D] \simeq \bigoplus_{\lambda} c_{\lambda}V[\lambda + (1^k)]$, where $\lambda + (1^k)$ denotes componentwise addition of partitions.
\end{lem}

\begin{proof}
It suffices to assume $\dim V = k$. Let $E = D \setminus c$. In this case, $s_{\lambda + (1^k)}(x_1, \ldots, x_k) = x_1 \cdots x_k s_{\lambda}(x_1, \ldots, x_k)$, so we want to show that $V[D] \simeq V[c] \otimes V[E]$. Proposition~\ref{prop:diagram-facts}(b) says that $V[c] \otimes V[E] = V[E \ast c]$. Write $c = \{c_1, \ldots, c_k\}$ and let $R_i$ be the row of $c_i$ in $E$. Let $\alpha_i = 1 + \sum_{b \in R_i} (c_i\,b) \in S_D$, and $\alpha = \alpha_1 \cdots \alpha_k$. Observe that $C(D) = C(E \ast c)$, whereas the summands of $\alpha$ form a set of right coset representatives for $R(E \ast c) \subseteq R(D)$. It follows that $y_D = y_{E \ast c} \alpha$ and hence $V[D] = V[E \ast c]\alpha$.

Let us see that multiplication by each $\alpha_i$ is injective on $V[E \ast c]$. Any element of $V[E \ast c] = V[c] \otimes V[E]$ is fixed by the right action of $R(E \ast c) = S_{R_1} \times \cdots \times S_{R_k}$, and $V[R_i]$ is the space of symmetric tensors in $V^{\otimes R_i}$, so we have $V[c] \otimes V[E] \subseteq V[c] \otimes V[R_1] \otimes \cdots \otimes V[R_k]$. Thus, it suffices to see that $\alpha_i$ is injective on $V[c] \otimes V[R_i] = V[c \ast R_i]$, or equivalently that $V[c \ast R_i]\alpha_i = V[c \cup R_i]$ is isomorphic to $V[c \ast R_i]$. But this is easy to check at the level of characters using Pieri's rule; say $|R_i| = r$:
\begin{equation*}
\ch V[c \ast R_i] = s_{(1^k)}(x_1, \ldots, x_k) s_{(r)}(x_1, \ldots, x_k) = s_{(r) + (1^k)}(x_1, \ldots, x_k) = \ch V[c \cup R_i].
\end{equation*}

Each $\alpha_i$ is a Jucys-Murphy element for $S_{R_i} \subseteq S_{R_i \cup \{c_i\}}$, and is shown in \cite{vershik-okounkov} to be diagonalizable as an operator on $\CC[S_{R_i \cup \{c_i\}}]$, hence on $V[E \ast c] \subseteq V^{\otimes E \ast c}$. It follows that $\ker \alpha|_{V[E \ast c]} = \sum_i \ker \alpha_i|_{V[E \ast c]}$, so $\alpha$ is injective on $V[E \ast c]$. We conclude $V[D] = V[E \ast c]\alpha \simeq V[E \ast c]$ as desired.
\end{proof}

A \emph{full column} of a three-row diagram is a column with three cells. If $c_1, \ldots, c_p$ are full columns in $D$ and $V[D \setminus (c_1 \cup \cdots \cup c_p)] \simeq \bigoplus_{\lambda} c_{\lambda} V[\lambda]$, then Lemma~\ref{lem:full-column} shows that $V[D] \simeq \bigoplus_{\lambda} c_{\lambda} V[\lambda + (p^k)]$. Thus for the purpose of decomposing $V[D]$ into irreducibles we can first delete all full columns.
\begin{lem} If $D$ is a three-row diagram with no full columns, then $D$ is equivalent to a toric skew shape. \end{lem}
\begin{proof}
There are 6 possible non-empty columns of $D$, namely
\begin{equation*}
\begin{array}{c} \square \\ \cdot \\ \cdot \end{array} \quad
\begin{array}{c} \square \\ \square \\ \cdot \end{array} \quad
\begin{array}{c} \cdot \\ \square \\ \cdot \end{array} \quad
\begin{array}{c} \cdot \\ \square \\ \square \end{array} \quad
\begin{array}{c} \cdot \\ \cdot \\ \square \end{array} \quad
\begin{array}{c} \square \\ \cdot \\ \square \end{array}
\end{equation*}
We can assume $D \subseteq [3] \times [p]$ and that $D$ has no empty columns. Sort the columns of $D$ so that all columns of the first type precede those of the second type, which precede those of the third type, etc. Now replace each column $\{(1,i), (3,i)\}$ (a column of the sixth type above) by $\{(3,i), (4,i)\}$. The image $E$ in $\Cyl_{3,p}$ of the resulting diagram is a toric skew shape, and on the other hand the image of $E$ in the torus $\ZZ/3\ZZ \times \ZZ/p\ZZ$ is equivalent to $D$.
\end{proof}

\bibliographystyle{alpha}
\bibliography{../bib/algcomb}

\end{document}